\documentclass[12pt]{amsart} 
\usepackage{amsmath}
\usepackage{paralist}
\usepackage[misc]{ifsym}

\usepackage[pagewise]{lineno}

\usepackage{epsfig} 
\usepackage{epstopdf} 
\usepackage[colorlinks=true]{hyperref}
\hypersetup{urlcolor=blue, citecolor=red}
\allowdisplaybreaks

\newtheorem{theorem}{Theorem}[section]
\newtheorem{corollary}[theorem]{Corollary}

\newtheorem{lemma}[theorem]{Lemma}

\theoremstyle{definition}

\title[Weak unique continuation property]
{Approximation of Elliptic Equations with Interior Single-Point Degeneracy and Its Application to Weak Unique Continuation Property
	} 

\author[Weijia Wu, Yaozhong Hu,   Donghui Yang and Jie Zhong]{}

\subjclass{Primary: 35J70; Secondary: 35B60.}
\keywords{Degenerate elliptic equations, Weak unique continuation property.}

\thanks{The second author is supported by the Natural Sciences and Engineering Research Council of Canada (NSERC RGPIN 2024-05941), and a centennial fund of the University of Alberta; The third author is supported by the National Natural Science Foundation of China.}

\thanks{$^*$Corresponding author: Jie Zhong}

\begin{document}
\maketitle

\centerline{\scshape
Weijia Wu,$^{{\href{mailto:weijiawu@yeah.net}{\textrm{\Letter}}}1}$
Yaozhong Hu,$^{{\href{mailto:yaozhong@ualberta.ca}{\textrm{\Letter}}}2}$
Donghui Yang,$^{{\href{mailto:donghyang@outlook.com}{\textrm{\Letter}}}}$
and Jie Zhong$^{{\href{mailto:jiezhongmath@gmail.com}{\textrm{\Letter}}}*1}$}

\medskip

{\footnotesize
 \centerline{$^1$School of Mathematics and Physics, North China Electric Power University, Beijing, 102206, China}
} 

\medskip

{\footnotesize
 \centerline{$^2$Department of Mathematical and Statistical Sciences, University of Alberta, Edmonton, AB T6G 2G1, Canada}
}

\medskip

{\footnotesize
	\centerline{$^3$School of Mathematics and Statistics, Central South University, Changsha, 410083, China}
}

\medskip

{\footnotesize
	\centerline{$^4$Department of Mathematics, California State University Los Angeles, Los Angeles, 90032, USA}
}

\bigskip

 \centerline{(Communicated by Handling Editor)}


\begin{abstract}
	This paper investigates the  weak unique continuation property (WUCP) for a class of high-dimensional elliptic equations with interior point degeneracy. First, we establish well-posedness results in weighted function spaces. Then, using an innovative approximation method, we derive the three sphere inequality at the degenerate point. Finally, we apply the three sphere inequality to prove WUCP for two different cases.
\end{abstract}


	\section {Introduction}

\label{sec:int}
\indent
The unique continuation properties for uniformly elliptic equations have been extensively studied in the literature (\cite{Adolfsson,Bakri,Banerjee,Carleman,Cavalheiro,Cavalheiro1,Fabes,Garofalo,Garofalo1987,Garofalo1,Hormander,Kenig,Kukavica1,Kukavica,Phung,Plis}). There are two types of unique continuation properties: the strong unique continuation property (SUCP) and the weak unique continuation property (WUCP). Below, we briefly recall these two properties.

Let $P(x,\partial)$ be a uniformly elliptic operator. The strong unique continuation property (SUCP) states that if $P(x,\partial)u = 0$ in a domain $\Omega \subset \mathbb{R}^N$, and there exists a point $x_0 \in \Omega$ such that $u$ vanishes to infinite order, meaning  
\begin{equation*}
	\int_{B_r(x_0)} u^2  \mathrm d x = O(r^k) \text{ as } r \to 0, \mbox{ for every } k \in \mathbb{N},
\end{equation*}		
where $B_r(x_0)$ denotes a ball in $\Omega$ centered at $x_0$ with radius $r$, then $u \equiv 0$ in $\Omega$.

The weak unique continuation property (WUCP) states  if $P(x,\partial)u = 0$ in $\Omega$, and $u = 0$ in an open subset $\omega \subset \Omega$, then $u \equiv 0$ in $\Omega$. It is easy to see that the WUCP requires less stringent conditions compared to the SUCP.

Notably, unique continuation does not hold universally for all uniformly elliptic equations (\cite{Plis}). Furthermore, the analysis of unique continuation properties becomes significantly more challenging for degenerate elliptic equations compared to uniformly elliptic ones. Currentlty, there are two methods — the three sphere inequality (\cite{Adolfsson, Garofalo, Kukavica1, Kukavica, Kukavica2,Phung,Vessella}) and Carleman estimates (\cite{Bakri, Banerjee, Carleman,Hormander, Kenig,  Plis}), which are effective in dealing with certain special cases (\cite{Banerjee, Garofalo3,Garofalo4, Garofalo1,  Wu1}).

The three sphere inequality states that for a harmonic (or subharmonic) function $u(x)$ defined in a region containing three concentric balls $B_{r_1}$, $B_{r_2}$ and $B_{r_3}$ with $r_1 < r_2 < r_3$, the maximum value $H(r)$ of $u(x)$  on the intermediate sphere $B_{r_2}$ can be bounded by a weighted geometric mean of the maximum values on the inner and outer spheres:
\begin{equation*}
	H(r_2)\leq (H(r_1))^\mu(H(r_3))^{1-\mu}, 
\end{equation*}
where $\mu\in(0,1)$ is determined by the radii. The three sphere inequality is developed on the basis of the doubling inequality, which was originally introduced by Garofalo and Lin in \cite{Garofalo}. 
The authors in \cite{Phung} provides a detailed introduction to the doubling inequality, the three sphere inequality, and their applications in unique continuation. 
In \cite{Garofalo4}, the authors primarily investigates the unique continuation properties of a specific class of second-order elliptic operators that degenerate on manifolds of arbitrary codimension, using the doubling inequality. The focus is on the model operator
\begin{equation*}
	P_\alpha = \Delta_z + |z|^{2\alpha} \Delta_t, \ \alpha>0,
\end{equation*}
in $\mathbb{R}^n \times \mathbb{R}^m$, which is elliptic outside a degeneracy manifold $(\left\lbrace 0\right\rbrace \times \mathbb{R}^m)$ but degenerates on it.
The authors establishes SUCP using Carleman estimates and introduces a quantitative version of SUCP that bypasses Carleman estimates, instead relying on the doubling inequality. Similarly, in \cite{tao2008weighted}, the doubling inequality is also applied to study the unique continuation properties of solutions to degenerate Schr\"odinger equations influenced by singular potentials and weighted settings. 
In \cite{banerjee2020strong}, SUCP is established for a class of degenerate elliptic operators with Hardy-type potentials using Carleman estimates. This work extends the results of \cite{Garofalo4} but does not yield a quantitative conclusion. Notably, the three sphere inequality appears to be more effective for studying quantitative weak unique continuation properties.

In this paper, we shall consider the   weak unique continuation properties for the elliptic equation with degenerate interior point by approximation. 
It is well known that the solution spaces of degenerate elliptic equations belong to weighted Sobolev spaces (\cite{Cavalheiro, Cavalheiro1, Fabes, Franchi, stuart2018stability}). A natural approach is to approximate a solution of a degenerate elliptic equation by a sequence of solutions to uniformly elliptic equations  (\cite{Cavalheiro, Cavalheiro1}). This method is feasible in weighted spaces and applies to high-dimensional cases, but it heavily relies on the Calder\'on-Zygmund decomposition, which can compromise certain desirable properties of the weight function. For instance, the approximating weight functions may lack differentiability, which is crucial when using the three sphere inequality to prove WUCP, requiring the approximating weight functions to be at least Lipschitz continuous. 
Another approximation approach, similar to that in \cite{MR2106129,cannarsa2016,wu2020carleman}, involves constructing a non-degenerate coefficient $|x+\epsilon|^\alpha$ over the entire domain $\Omega$ to approximate the degenerate coefficient $|x|^\alpha$. However, this method is suitable for one-dimensional degenerate equations but not for the high-dimensional problems we aim to study. For the problem we consider in this paper, local estimates are required to approximate the solution (see Lemma \ref{08.15.L5}).

While the idea of approximation has been utilized in many works, our method is fundamentally different from those in the existing literature.
First, one of our main contributions is the introduction of an alternative approximation method for a specific class of weight functions with a single degenerate interior point. Our approximation is achieved by constructing a carefully designed non-degenerate weight function to approximate the degenerate weight function within a small local region $B_\epsilon$ rather than the entire domain $\Omega$. This ensures that the weight function remains differentiable in high-dimensional settings. For a detailed discussion, refer to Section \ref{S3}.
Second, in the proof of WUCP, we consider two cases: 
$0\in \omega$ and $0\notin \omega$. For case $0\in \omega$, 
the result is obtained using the three sphere inequality at both degenerate and non-degenerate points. For the more challenging case $0\notin \omega$, we apply Schauder estimates to address the difficulties arising from the degenerate point being excluded from $\omega$. 		 
Finally, we derive a WUCP result. 

It is worth noting that in most works (see \cite{MR2106129,cannarsa2016,wu2020carleman}), the SUCP is typically achieved using the doubling inequality. However, this paper employs the more robust three sphere inequality. Although we do not present results on SUCP here, we have demonstrated it in another working paper using an annular estimate method. 

We organize the paper as follows: In Section \ref{S2}, we present several well-posedness results. In Section \ref{S3}, we provide a detailed explanation of the construction of the approximation and introduce the preliminary lemmas required for proving the three sphere inequality at the degenerate point. In Section \ref{S4}, we establish the three sphere inequality at the degenerate point and prove WUCP for two cases:$0\in \omega$ and $0\notin \omega$.

\section{Preliminary results}\label{S2}
Let us consider the following equation
\begin{equation}\label{08.15.1}
	\begin{cases}
		-\mathrm{div}(w \nabla u)=f, &\mbox{in } \Omega, \\
		u=0, &\mbox{on } \partial\Omega, 
	\end{cases}
\end{equation}
where $\Omega \subset \mathbb{R}^N$ ($N\ge 2$) is a domain containing the origin ($0 \in \Omega$), and its boundary $\partial \Omega$ is of class $C^2$. The weight function is given by $w = |x|^\alpha$, with a fixed $\alpha \in (0,2)$, and $f$ is a given function such that  $f \in L^2(\Omega; w^{-1})$. The weighted Sobolev space  $L^2(\Omega; w)$  defined for every $w>0$  almost everywhere as:
\begin{equation*}
	L^2(\Omega; w)=\left\lbrace u(x) \mid	u \mbox{ is measurable, and } \int_\Omega u^2w\mathrm d x<\infty \right\rbrace .
\end{equation*}
The inner product on $L^2(\Omega;w)$ is 
\begin{equation*}
	(u,v)_{L^2(\Omega;w)}=\int_\Omega uvw\mathrm d x, 
\end{equation*}
and the norm on $L^2(\Omega; w)$ is 
\begin{equation*}
	\|u\|_{L^2(\Omega;w)}=\left(\int_\Omega u^2w\mathrm d x\right)^ \frac{1}{2}. 
\end{equation*}
It is well known that $(L^2(\Omega; w),(\cdot,\cdot)_{L^2(\Omega;w)})$  (see \cite{GC}) is a Hilbert space and $(L^2(\Omega; w), \|\cdot\|_{L^2(\Omega;w)})$ is a Banach space.



Set 
\begin{equation*}
	H_w^{1}(\Omega)=\left\{u\in L^2(\Omega)\colon  \frac{\partial u}{\partial x_i}\in L^2(\Omega;w), i=1,\cdots, N\right\}, 
\end{equation*}
where $ \frac{\partial u}{\partial x_i}, i=1,\cdots, N$ are the distribution partial derivatives, the inner product on $H_w^1(\Omega)$ is 
\begin{equation*}
	\begin{split}
	(u, v)_{H_w^1(\Omega)}&=\int_\Omega uvw\mathrm d x+\sum_{i=1}^N \int_\Omega \frac{\partial u}{\partial x_i} \frac{\partial v}{\partial x_i}w\mathrm d x\\
	 &= (u,v)_{L^2(\Omega;w)} + (\nabla u , \nabla v)_{L^2(\Omega;w)}
	\end{split}
\end{equation*}
and the norm is 
\begin{equation*}
	\|u\|_{H_w^1(\Omega)}=\left(\int_\Omega u^2w\mathrm d x+\sum_{i=1}^N\int_\Omega \left| \frac{\partial u}{\partial x_i}\right|^2w\mathrm d x\right)^ \frac{1}{2}. 
\end{equation*}
Define
\begin{equation*}
	H_{w,0}^1(\Omega)=\overline{\mathcal{D}(\Omega)}^{\|\cdot\|_{H_{w}^1(\Omega)}},
\end{equation*}
where $\mathcal{D}(\Omega) = C_0^\infty(\Omega)$ is the space of test functions.
Denote by $H_w^{-1}(\Omega)$ the dual space of $H_{w,0}^1(\Omega)$. This space is a subspace of $\mathcal{D}'(\Omega)$, the space of distributions on $\Omega$. It is well known that $(H_w^1(\Omega), (\cdot, \cdot)_{H_w^1(\Omega)})$ forms a Hilbert space, while $(H_w^1(\Omega), \|\cdot\|_{H_w^1(\Omega)})$ is a Banach space.

Next, we aim to establish some well-posedness results for equation \eqref{08.15.1}. 
First, we introduce some notations that will be used:
\begin{equation*}
	\Omega^\epsilon= \left\lbrace x\in \Omega \mid |x|>\epsilon\right\rbrace, \ B_\epsilon= \left\lbrace x\in \Omega \mid |x|<\epsilon\right\rbrace.
\end{equation*}
Similar to the proof of Lemma 3.1 in \cite{Wu} or Proposition 2.1 (1) in \cite{stuart2018stability}, we can easily derive the following weighted Hardy inequality.

\begin{lemma}\label{08.15.L1}
	For any $N \geq 2$ and $\alpha \in (0,2)$, if $u \in H_{w,0}^1(\Omega)$, then we have
	\begin{equation}\label{hardy61}
		(N-2+\alpha) \left( \int_{\Omega} |x|^{\alpha-2} u^2 \, \mathrm dx \right) ^{\frac{1}{2}} \leq 2 \left( \int_{\Omega} |x|^{\alpha} \nabla u \cdot \nabla u \, \mathrm dx \right) ^{\frac{1}{2}}.
	\end{equation}
	Moreover, let $m := \sup_{x \in \Omega} |x| + 1$. We can also derive  
	\begin{equation}\label{hardy62}
		\left( \int_{\Omega} u^2 \, \mathrm dx \right) ^{\frac{1}{2}}\leq \frac{2m^{1-\frac{\alpha}{2}}}{N-2+\alpha}\left( \int_{\Omega} |x|^{\alpha} \nabla u \cdot \nabla u \, \mathrm dx \right) ^{\frac{1}{2}}.
	\end{equation}						
\end{lemma}
\begin{proof}
	Let $u \in H_{w,0}^1(\Omega)$. Then $u = 0$ on $\partial \Omega$ in the sense of traces. Moreover, for each $\epsilon> 0$, the restriction of $u$ to the open set $\Omega^\epsilon$ belongs to $W^{1,2}(\Omega^\epsilon)$, and the trace of $u$ on $\partial \Omega^\epsilon$ is well-defined and bounded in $L^2(\partial \Omega^\epsilon)$.
	Applying the divergence theorem on $\Omega^\epsilon$, we obtain 
	\begin{equation*}
		\begin{aligned}
			2 \int_{\Omega^\epsilon} |x|^{\alpha-2} (x \cdot \nabla u) u \, \mathrm dx &= \int_{\Omega^\epsilon} |x|^{\alpha-2} x \cdot \nabla (u^2) \, \mathrm dx \\
			&= \int_{\partial \Omega} |x|^{\alpha-2} u^2 (x \cdot \nu) \, \mathrm ds + \int_{\partial B_\epsilon} |x|^{\alpha-2} u^2 (x \cdot \nu) \, \mathrm ds \\
			&\quad - \int_{\Omega^\epsilon} \operatorname{div}\left( |x|^{\alpha-2} x \right) u^2 \, \mathrm dx.
		\end{aligned}
	\end{equation*} 
	
	Since u = 0 on $\partial \Omega$, the boundary integral over $\partial \Omega$ vanishes. Noting that $x \cdot \nu = -|x|$ on $\partial B_\epsilon$ and
	$\operatorname{div}\left( |x|^{\alpha-2} x \right) = (N-2+\alpha) |x|^{\alpha-2}$,
	we deduce
	\begin{equation*}
		\begin{aligned}
			2 \int_{\Omega^\epsilon} |x|^{\alpha-2} (x \cdot \nabla u) u \, \mathrm dx
			= -\int_{\partial B_\epsilon} |x|^{\alpha-1} u^2 \, \mathrm ds - (N-2+\alpha) \int_{\Omega^\epsilon} |x|^{\alpha-2} u^2 \, \mathrm dx.
		\end{aligned}
	\end{equation*}
	Rearranging terms gives
	\begin{equation*}
		(N-2+\alpha) \int_{\Omega^\epsilon} |x|^{\alpha-2} u^2 \, \mathrm dx \leq -2 \int_{\Omega^\epsilon} |x|^{\alpha-2} (x \cdot \nabla u) u \, \mathrm dx.
	\end{equation*}
	Applying the Cauchy–Schwarz inequality, we estimate
	\begin{equation*}
		\begin{aligned}
			&\hspace{4.5mm}-2 \int_{\Omega^\epsilon} |x|^{\alpha-2} (x \cdot \nabla u) u \, \mathrm dx \leq 2 \int_{\Omega^\epsilon} |x|^{\frac{\alpha}{2}-1} |u| \, |x|^{\frac{\alpha}{2}} |\nabla u| \, \mathrm dx \\
			&\leq 2 \left( \int_{\Omega^\epsilon} |x|^{\alpha-2} u^2 \, \mathrm dx \right)^{1/2} \left( \int_{\Omega^\epsilon} |x|^{\alpha} |\nabla u|^2 \, \mathrm dx \right)^{1/2}.
		\end{aligned}
	\end{equation*}
	Thus, we obtain
	\begin{equation}\label{hardy_eps}
		(N-2+\alpha) \left( \int_{\Omega^\epsilon} |x|^{\alpha-2} u^2 \, \mathrm dx \right)^{1/2} \leq 2 \left( \int_{\Omega^\epsilon} |x|^{\alpha} |\nabla u|^2 \, \mathrm dx \right)^{1/2}.
	\end{equation}
	
	Passing to the limit as $\epsilon\to 0^+$, we use the almost everywhere convergence of $|x|^{\alpha-2} u^2 \chi_{\Omega^\epsilon} \to |x|^{\alpha-2} u^2$, and Fatou’s lemma yields
	\begin{equation*}
		\int_{\Omega} |x|^{\alpha-2} u^2 \, \mathrm dx \leq \liminf_{\epsilon\to 0} \int_{\Omega^\epsilon} |x|^{\alpha-2} u^2 \, \mathrm dx.
	\end{equation*}
	Since $|x|^\alpha |\nabla u|^2 \in L^1(\Omega)$,  and taking limits in \eqref{hardy_eps} yields
	\begin{equation*}
		(N-2+\alpha) \left( \int_{\Omega} |x|^{\alpha-2} u^2 \, \mathrm dx \right)^{1/2} \leq 2 \left( \int_{\Omega} |x|^{\alpha} |\nabla u|^2 \, \mathrm dx \right)^{1/2},
	\end{equation*}
	thus establishing inequality \eqref{hardy61}.
	
	Finally, to prove inequality \eqref{hardy62}, observe that
	\begin{equation*}
		|u(x)|^2 \leq m^{2-\alpha} |x|^{\alpha-2} |u(x)|^2, \quad \forall \, x \in \Omega,
	\end{equation*}
	where $m := \sup_{x\in \Omega} |x| + 1$. Integrating over $\Omega$ and applying \eqref{hardy61} gives
	\begin{equation*}
		\left( \int_{\Omega} u^2 \, \mathrm dx \right)^{1/2} \leq \frac{2m^{1-\alpha/2}}{N-2+\alpha} \left( \int_{\Omega} |x|^{\alpha} |\nabla u|^2 \, \mathrm dx \right)^{1/2}.
	\end{equation*}
	The proof is thus complete.
\end{proof}		

From Lemma \ref{08.15.L1}, it is  also evident that space $H_{w,0}^1(\Omega)$ is embedded into space 
$L^2(\Omega)$. Next, we will prove that this embedding is compact.		
\begin{lemma}\label{08.15.L4}
	The embedding $H_{w,0}^1(\Omega)\subset \subset  L^2(\Omega)$ is compact. 
\end{lemma}
\begin{proof}
	To establish the compactness of the embedding it suffices to show that if $\left\{u_n\right\}$ is a sequence converging weakly to zero in $H_{w,0}^1(\Omega)$ as $n \rightarrow \infty$, then $\left\|u_n\right\|_{L^2(\Omega)} \rightarrow 0$ as $n \rightarrow \infty$. 
	
	Since $H_{w,0}^1(\Omega)$ is continuously embedded in $L^2(\Omega)$ by Lemma \ref{08.15.L1}, $L^2(\Omega)^* \subset$ $H_{w,0}^1(\Omega)^*$ and hence $\left\{u_n\right\}$ converges weakly to zero in $L^2(\Omega)$. 
	
	Consider $\epsilon>0$. If $\left\{u_n\right\}$ does not converge weakly to zero in $W^{1,2}\left(\Omega^{\epsilon}\right)$, there exist $f \in W^{1,2}\left(\Omega^{\epsilon}\right)^*$, a subsequence $\left\{u_{n_k}\right\}$ and $\delta>0$ such that $\left|f\left(u_{n_k}\right)\right| \geq \delta$ for all $n_k$. Passing to a further subsequence if necessary, we can suppose that $\left\{u_{n_k}\right\}$ converges weakly to an element $v$ in $W^{1,2}\left(\Omega^{\epsilon}\right)$. Thus $\left\{u_{n_k}\right\}$ converges weakly to $v$ in $L^2\left(\Omega^{\epsilon}\right)$ and so $v=0$ a.e. on $\Omega^{\epsilon}$ since $\left\{u_n\right\}$ converges weakly to zero in $L^2(\Omega)$ and hence also on $L^2\left(\Omega^{\epsilon}\right)$. But then $f\left(u_{n_k}\right) \rightarrow f(v)=f(0)=0$ as $n_k \rightarrow \infty$, contradicting with $\left|f\left(u_{n_k}\right)\right| \geq \delta$ for all $n_k$. Hence $\left\{u_n\right\}$ converges weakly to zero in $W^{1,2}\left(\Omega^{\epsilon}\right)$ and therefore $\left\|u_n\right\|_{L^2\left(\Omega^{\epsilon}\right)} \rightarrow 0$ as $n \rightarrow \infty$ . From the above it follows that
	\begin{equation}\label{UN}
		\limsup _{n \rightarrow \infty}\left\|u_n\right\|_{L^2(\Omega)}^2=\limsup _{n \rightarrow \infty} \int_{B_\epsilon}\left|u_n\right|^2 d x.
	\end{equation}
	But from \cite{MR1794994} and \eqref{hardy62} in Lemma \ref{08.15.L1}, we have
	\begin{equation}\label{LP}
		\left\|u \right\|_{L^q(\Omega)} \le C \left\| u\right\|_{H_{w,0}^1(\Omega)}, \ 1\le q \le \frac{2N}{N-2+\alpha},
	\end{equation}
	then (taking $q>2$)
	\begin{equation*}
		\int_{B_\epsilon}\left|u_n\right|^2 d x \le \left(\int_{B_\epsilon} |1|^ {\frac{q}{q-2}} dx\right) ^{\frac{q-2}{q}} \left( \int_{B_\epsilon} \left( |u_n|^2 \right) ^{\frac{q}{2}}dx\right)^{\frac{2}{q}}
		\le |B_\epsilon|^ {\frac{q-2}{q}} \left\| u_n \right\| ^2_{L^q(\Omega)}.
	\end{equation*}
	The weak convergence of $\left\{u_n\right\}$ in $H_{w,0}^1(\Omega)$ and \eqref{LP} imply that this sequence is bounded in $L^q(\Omega)$, then (note that $q>2$)
	$$
	\int_{B_\epsilon}\left|u_n\right|^2 d x \le C |B_\epsilon|^\frac{q-2}{q}\left\| u_n\right\|^{2}_{H_{w,0}^1(B_\epsilon)}\leq C |B_\epsilon|^\frac{q-2}{q}.
	$$
	Letting $\epsilon\rightarrow 0^+$ in \eqref{UN} shows that $\left\|u_n\right\|_{L^2(\Omega)} \rightarrow 0$ as $n \rightarrow \infty$, completing the proof.
\end{proof}
Next, we use the Lax-Milgram theorem to show that equation \eqref{08.15.1} has a unique weak solution $u\in H_{w,0}^1(\Omega)$ in the sense of 
\begin{equation*}
	\int_\Omega (\nabla u\cdot \nabla v) w\mathrm d x=\int_\Omega fv\mathrm d x
\end{equation*}
for all $v\in H_{w,0}^1(\Omega)$. 
\begin{lemma}\label{08.15.L2}
	For each $f\in L^2(\Omega;w^{-1})$, there exists a unique solution for the equation \eqref{08.15.1}. 
\end{lemma}

\begin{proof}
	Denote 
	\begin{equation*}
		B_w[u,v]=\int_\Omega (\nabla u\cdot \nabla v)w\mathrm d x \mbox{ for all } u, v\in H_{w,0}^1(\Omega). 
	\end{equation*}
	It is easily verified that $B_w[\cdot, \cdot]: H_{w,0}^1(\Omega)\times H_{w,0}^1(\Omega)\rightarrow \mathbb{R}$ be a bilinear form. 
	
	On one hand, we have 
	\begin{equation*}
		\left|B_w[u,v]\right|\leq \|u\|_{H_{w,0}^1(\Omega)}\|v\|_{H_{w,0}^1(\Omega)}, 
	\end{equation*}
	and $B_w[u,u]=\|u\|_{H_{w,0}^1(\Omega)}$. On the other hand,  we have 
	\begin{equation*}
		\left|\int_\Omega f v\mathrm d x\right| \leq \|f\|_{L^2(\Omega; w^{-1})}\|v\|_{L^2(\Omega; w)}\leq C\|f\|_{L^2(\Omega; w^{-1})}\|v\|_{H_{w,0}^1(\Omega)}
	\end{equation*}
	by Cauchy inequality and \eqref{hardy62} in Lemma \ref{08.15.L1}, i.e., $f: H_{w,0}^1(\Omega)\rightarrow\mathbb{R}$ is a bounded linear functional on $H_{w,0}^1(\Omega)$. 
	
	Finally, by the Lax-Milgram theorem, we obtain that there exists a unique $u\in H_{w,0}^1(\Omega)$ satisfying \eqref{08.15.1}. 
\end{proof}

\section{Approximations}\label{S3}
Our approach is to approximate the solution of a degenerate equation by a sequence of solutions to non-degenerate equations that satisfy the uniform ellipticity condition.


Let 
\begin{equation}\label{we}
	w_\epsilon=
	\begin{cases}
		|x|^\alpha, & |x|\geq \epsilon,\\
		( \frac{3}{4}|x|^2+ \frac{1}{4}\epsilon^2)^ \frac{\alpha}{2}, &|x|\leq \epsilon. 
	\end{cases}
\end{equation}
Then it is clear that $w_\epsilon\in C^{0,1}(\overline \Omega)$ since $\alpha\in (0,2)$, $w_\epsilon$ is a radial convex function on $\mathbb{R}^N$ and nondecreasing on $[0,\infty)$, and $( \frac{\epsilon}{2})^\alpha\leq w_\epsilon\leq \epsilon^\alpha$ in $B_\epsilon$, and 
\begin{equation*}
	\nabla w_\epsilon=
	\begin{cases}
		\alpha |x|^{\alpha-2}x, & |x|> \epsilon,\\
		\alpha\left( \frac{3}{4}|x|^2+ \frac{1}{4}\epsilon^2\right)^{ \frac{\alpha}{2}-1} \frac{3}{4}x, &|x|< \epsilon. 
	\end{cases}
\end{equation*}
It is worth noting that,  our approximation methods is different from the one used in other literature (see \cite{MR2106129,cannarsa2016,wu2020carleman}), such as setting a non-degenerate coefficient $|x+\epsilon|^\alpha$ to approximate the degenerate coefficient $|x|^\alpha$, which takes the form $|x+\epsilon|^\alpha$ over the entire domain $\Omega$. However, in this paper, our setup of $w_\epsilon$ ensures that the approximate  coefficients do not depend on $\epsilon$ outside $B_\epsilon$ while  approximating the original degenerate coefficient within $B_\epsilon$. This allows us to achieve better estimates of the solution and obtain improved regularity results, even in the high-dimensional case.	

For each $k\in\mathbb{N}$, we denote   $w_{ \frac{1}{k}}$ by $w_k$  and  
consider the following approximate equation
\begin{equation}\label{08.15.3}
	\begin{cases}
		-\mathrm{div}(w_k \nabla u_k)=f_k, &\mbox{in }\Omega,\\
		u_k=0, &\mbox{on }\partial \Omega
	\end{cases}
\end{equation}
with $f_k\in L^2(\Omega; w_k^{-1})$. 
We say that $u_k\in H_{w_k,0}^1(\Omega)$ is a {\it weak solution} of \eqref{08.15.3}, if 
\begin{equation*}
	\int_\Omega (\nabla u_k\cdot \nabla v)w_k\mathrm d x=\int_\Omega f_k v\mathrm d x
\end{equation*}
for all $v\in H_{w_k,0}^1(\Omega)$. 

We note that $H_{w_k,0}^1(\Omega)=H_0^1(\Omega)$ since $( \frac{1}{k})^\alpha\leq w_k\leq m^\alpha$ ($m := \sup_{x\in \Omega} |x| + 1$ defined in Lemma \ref{08.15.L1}) for each $k\in\mathbb{N}$, where $H_0^1(\Omega)$ is the classical Sobolev spaces. 

As Lemma \ref{08.15.L1}, we provide a proof of the Hardy inequality for the non-degenerate equation.
\begin{lemma}\label{08.16.L2}
	Let $u\in H_{w_k,0}^1(\Omega)$. Then 
	\begin{equation}\label{08.20.1}
		(N+\alpha-2)\|w_k^{ \frac{1}{2}- \frac{1}{\alpha}}u\|_{L^2(\Omega)}\leq 2\|u\|_{H_{w_k,0}^1(\Omega)}. 
	\end{equation}
	Moreover, we have 
	\begin{equation}\label{08.20.2}
		\|u\|_{L^2(\Omega; w_k)}\leq  \frac{2m}{N+\alpha-2}\||\nabla u|\|_{L^2(\Omega;w_k)}.
	\end{equation}
\end{lemma}

\begin{proof}
	Denote $\epsilon= \frac{1}{k}$. We shall prove 
	\begin{equation*}
		(N+\alpha-2)\|u\|_{L^2(\Omega; w_\epsilon)}\leq 2\|u\|_{H_{w_\epsilon,0}^1(\Omega)}
	\end{equation*}
	for each $u\in H_{w_\epsilon,0}^1(\Omega)$. 
	
	Since $w_\epsilon\in C^{0,1}(\overline\Omega)$ (i.e., $w_\epsilon\in W^{1,\infty}(\Omega)$), we have
	\begin{equation*}
		\begin{split}
			&\hspace{4.5mm}2\int_{\Omega} w_\epsilon^{1- \frac{2}{\alpha}}u(x\cdot\nabla u)\mathrm d x
			=\int_\Omega w_\epsilon^{1- \frac{2}{\alpha}}x\cdot \nabla u^2\mathrm d x\\
			&=\int_\Omega \mathrm{div}\left(w_\epsilon^{1- \frac{2}{\alpha}}u^2x\right)\mathrm d x-\int_\Omega u^2\mathrm{div}\left(w_\epsilon^{1- \frac{2}{\alpha}}x\right)\mathrm d x\\ 
			&=-(N+\alpha-2)\int_\Omega u^2w_\epsilon^{1- \frac{2}{\alpha}} \mathrm d x- \frac{2-\alpha}{4}\epsilon^2\int_{B_\epsilon}\left( \frac{3}{4}|x|^2+ \frac{1}{4}\epsilon^2\right)^{ \frac{\alpha}{2}-2} u^2\mathrm d x. 
		\end{split}
	\end{equation*}
	Then 
	\begin{equation*}
		\begin{split}
			&\hspace{4.5mm}(N+\alpha-2)\int_\Omega u^2w_\epsilon^{1- \frac{2}{\alpha}}\mathrm d x\\
			&\leq -2\int_\Omega w_\epsilon^{1- \frac{2}{\alpha}}u(x\cdot \nabla u)\mathrm d x=-2\int_\Omega \left(w_\epsilon^{ \frac{1}{2}- \frac{1}{\alpha}}u\right)\left(w_\epsilon^{ \frac{1}{2}- \frac{1}{\alpha}}x\cdot\nabla u\right)\mathrm d x\\
			&\leq 2\left(\int_\Omega w_\epsilon^{1- \frac{2}{\alpha}}u^2\mathrm d x\right)^ \frac{1}{2}\left(\int_\Omega w_\epsilon^{1- \frac{2}{\alpha}}|x|^2|\nabla u|^2\mathrm d x\right)^ \frac{1}{2}\\
			&\leq 2\left(\int_\Omega w_\epsilon^{1- \frac{2}{\alpha}}u^2\mathrm d x\right)^ \frac{1}{2}\left(\int_\Omega |\nabla u|^2w_\epsilon\mathrm d x\right)^ \frac{1}{2}
		\end{split}
	\end{equation*}
	by $\alpha\in (0,2)$ and $|x|^2\leq  \frac{3}{4}|x|^2+ \frac{1}{4}\epsilon^2$ on $B_\epsilon$. This shows that 
	\begin{equation*}
		(N+\alpha-2)\|w_\epsilon^{ \frac{1}{2}- \frac{1}{\alpha}}u\|_{L^2(\Omega)}\leq 2\||\nabla u|\|_{L^2(\Omega; w_\epsilon)}. 
	\end{equation*}
	
	Finally, by \eqref{08.20.1} and  $ \frac{1}{4}\epsilon^2\leq  \frac{3}{4}|x|^2+ \frac{1}{4}\epsilon^2\leq m^2$ in $\Omega$, we get \eqref{08.20.2}. 
\end{proof}

To prove that the solution of the non-degenerate equation converges weakly in the solution space to the solution of the degenerate equation, we 
first show  that the approximate solutions are  bounded.

\begin{lemma}\label{08.15.L3}
	Let $u_k$ be a solution of \eqref{08.15.3} with $f_k\in L^2(\Omega;w_k^{-1})$. For each $k\in\mathbb{N}$, then
	\begin{equation*}
		\|u_k\|_{H_{w_k,0}^1(\Omega)}\leq C \|f_k\|_{L^2(\Omega;w_k^{-1})}, 
	\end{equation*}
	where the constant $C>0$ depends only on $\alpha, N$ and $\Omega$. 
\end{lemma}

\begin{proof}
	Let $u_k\in H_{w_k,0}^1(\Omega)$ be the test function. Then 
	\begin{equation*}
		\begin{split}
			\int_\Omega |\nabla u_k|^2w_k\mathrm d x\leq \left(\int_\Omega |f_k w_k^{-1}|^2w_k\mathrm d x\right)^ \frac{1}{2}\left(\int_\Omega u_k^2w_k\mathrm d x\right)^ \frac{1}{2}.
		\end{split} 
	\end{equation*} 
	By \eqref{08.20.2} in Lemma  \ref{08.16.L2}, we get 
	\begin{equation*}
		\|u_k\|_{H_{w_k,0}^1(\Omega)}\leq C\|f_k\|_{L^2(\Omega;w_k^{-1})}\,, 
	\end{equation*}
	where the constant $C>0$ depends only on $\alpha, N$ and $\Omega$. 
\end{proof}

Now, we are ready to show   the  existence of the solution for degenerate equation \eqref{08.15.1}
by using the  approximation via the solutions of the non-degenerate equations. 

\begin{lemma}\label{08.15.L5}
	Let $u_k\in H_{w_k,0}(\Omega)$ be the solution of \eqref{08.15.3} with $f_k=f$,  where $f\in L^2(\Omega;w^{-1})$  and $k\in\mathbb{N}$. Then,
	there is a $u_0\in H_{w,0}^1(\Omega)$ such that 
	\begin{equation}\label{08.15.4}
		u_k\rightharpoonup u_0 \mbox{ weakly in } H_{w,0}^1(\Omega), 
	\end{equation}
	and 
	\begin{equation}\label{08.15.12}
		u_k\rightarrow u_0 \mbox{ strongly in } L^2(\Omega). 
	\end{equation}
	Moreover,  $u_0 $ is  the unique solution of \eqref{08.15.1} with $f\in L^2(\Omega;w^{-1})$. 
\end{lemma}

\begin{proof}
	Since $f\in L^2(\Omega;w^{-1})$, we have 
	\begin{equation*}
		\int_\Omega (fw_k)^{-1}w_k\mathrm d x=\int_\Omega (fw^{-1})^2(ww_k^{-1})w\mathrm d x\leq \int_\Omega (fw^{-1})^2w\mathrm d x
	\end{equation*}
	according to $w\leq w_k$ for all $k\in\mathbb{N}$. i.e., $f\in L^2(\Omega; w_k^{-1})$ for all $k\in\mathbb{N}$. 
	From Lemma \ref{08.15.L3}, for each $k\in\mathbb{N}$, since $w\leq w_k$ for all $k\in\mathbb{N}$, we have 
	\begin{equation*}
		\|u_k\|_{H_{w,0}^1(\Omega)}\leq \|u_k\|_{H_{w_k,0}^1(\Omega)}\leq C\|f\|_{L^2(\Omega; w_k^{-1})}\leq C\|f\|_{L^2(\Omega; w^{-1})},
	\end{equation*}
	where the constant $C>0$ depends only on $\alpha, N$ and $\Omega$. 
	Then there exists a subsequence of $\{u_k\}_{k\in\mathbb{N}}$, still denote by itself, and $\widetilde u_0\in H_{w,0}^1(\Omega)$, such that 
	\begin{equation*}
		u_k\rightharpoonup \widetilde u_0 \mbox{ weakly in } H_{w,0}^1(\Omega), 
	\end{equation*}
	and 
	\begin{equation*}
		u_k\rightarrow \widetilde u_0 \mbox{ strongly in } L^2(\Omega). 
	\end{equation*}
	by $H_{w,0}^1(\Omega)\subset \subset  L^2(\Omega)$ is compact (see Lemma \ref{08.15.L4}). 
	
	Now, we prove $\widetilde u_0$ satisfies the equation \eqref{08.15.1} with $f\in L^2(\Omega;w^{-1})$. 
	
	Let $\psi\in \mathcal{D}(\Omega)$. Let $\epsilon>0$. By \eqref{08.15.4}, there exists $k_0\in\mathbb{N}$, such that 
	\begin{equation}\label{08.15.5}
		\left|\int_\Omega (\nabla u_k\cdot \nabla\psi)w\mathrm d x-\int_\Omega (\nabla \widetilde u_0\cdot \nabla \psi)w\mathrm d x\right|< \frac{1}{2}\epsilon\mbox{ when } k\geq k_0. 
	\end{equation} 
	Since $u_k$ is a solution of \eqref{08.15.3} for each $k\in\mathbb{N}$, we have 
	\begin{equation}\label{08.18.9}
		\int_\Omega (\nabla u_k\cdot \nabla \psi) w_k\mathrm d x=\int_\Omega f\psi\mathrm d x. 
	\end{equation} 
	Note that  $w_k=w$ on $\Omega\setminus B_{ \frac{1}{k}}$ for $k\geq k_0$, we have 
	\begin{equation}\label{08.15.6}
		\begin{split}
			\int_\Omega (\nabla u_k\cdot \nabla \psi)w_k\mathrm d x
			&=\int_{\Omega\setminus B_{ \frac{1}{k}}}(\nabla u_k\cdot \nabla\psi)w\mathrm d x+\int_{B_{ \frac{1}{k}}}(\nabla u_k\cdot \nabla\psi)w_k\mathrm d x.
		\end{split}
	\end{equation}
	By Lemma \ref{08.15.L3} we have 		
	\begin{equation}\label{08.15.8} 
		\begin{split}
			\left|\int_{B_{ \frac{1}{k}}}(\nabla u_k\cdot \nabla\psi)w_k\mathrm d x\right|
			&\leq \left(\int_{B_{ \frac{1}{k}}}|\nabla u_k|^2w_k\mathrm d x\right)^ \frac{1}{2}\left(\int_{B_{ \frac{1}{k}}}|\nabla\psi|^2w_k\mathrm d x\right)^ \frac{1}{2}\\
			&\leq \left(\sup_{x\in\Omega}|\nabla\psi(x)|\right)\left(\int_{\Omega}|\nabla u_k|^2w_k\mathrm d x\right)^ \frac{1}{2}w_k(B_{ \frac{1}{k}})^ \frac{1}{2}\\
			&\leq C_{\psi, N, \alpha} \|f\|_{L^2(\Omega;w^{-1})}  \frac{1}{k^{ \frac{\alpha+N}{2}}}, 
		\end{split}
	\end{equation}
	where $C_{\psi, N, \alpha}>0$ is a constant that depends only on $\psi, \Omega, N$ and $\alpha$, and $w_k(B_{ \frac{1}{k}})^ \frac{1}{2} = \left(\int_{B_{ \frac{1}{k}}}w_k \mathrm d x \right)^{ \frac{1}{2}}$. Hence we can assume 
	\begin{equation}\label{08.15.7}
		\left|\int_{B_{ \frac{1}{k}}}(\nabla u_k\cdot \nabla\psi)w_k\mathrm d x\right|< \frac{1}{4}\epsilon\mbox{ when } k\geq k_0\,. 
	\end{equation}
	On the other hand, by $w\leq w_k$,  and by the same argument as for \eqref{08.15.8}, we have 
	\begin{equation*}
		\begin{split}
			\left|\int_{B_{ \frac{1}{k}}}(\nabla u_k\cdot \nabla\psi)w\mathrm d x\right|
			&\leq \int_{B_{ \frac{1}{k}}}|\nabla u_k||\nabla \psi|w_k\mathrm d x\\
			&\leq \left(\int_{B_{ \frac{1}{k}}}|\nabla u_k|^2w_k\mathrm d x\right)^ \frac{1}{2}\left(\int_{B_{ \frac{1}{k}}}|\nabla\psi|^2w_k\mathrm d x\right)^ \frac{1}{2}\\
			&\leq C_{\psi, N,\alpha}\|f\|_{L^2(\Omega;w^{-1})} \frac{1}{k^{ \frac{\alpha+N}{2}}}, 
		\end{split}
	\end{equation*}
	hence we also can assume 
	\begin{equation}\label{08.15.9}
		\left|\int_{B_{ \frac{1}{k}}}(\nabla u_k\cdot \nabla\psi)w\mathrm d x\right|< \frac{1}{4}\epsilon\mbox{ when } k\geq k_0. 
	\end{equation}
	From \eqref{08.15.5}, \eqref{08.18.9}, \eqref{08.15.6}, \eqref{08.15.7} and \eqref{08.15.9},   we obtain 
	\begin{equation*}
		\begin{split}
			&\hspace{4.5mm}\left|\int_\Omega(\nabla \widetilde u_0\cdot \nabla \psi)w\mathrm d x-\int_\Omega f\psi\mathrm d x\right|\\
			&\leq \left|\int_\Omega (\nabla u_k\cdot \nabla \psi)w_k\mathrm d x-\int_\Omega(\nabla u_k\cdot \nabla\psi)w\mathrm d x\right|\\
			&\hspace{4.5mm}+\left|\int_\Omega (\nabla \widetilde u_0\cdot \nabla\psi)w\mathrm d x-\int_\Omega (\nabla u_k\cdot \nabla\psi)w\mathrm d x\right|\\
			&\hspace{4.5mm}+\left|\int_\Omega (\nabla u_k\cdot \nabla\psi)w_k\mathrm d x-\int_\Omega f\psi\mathrm d x\right|\\
			&\leq \left|\int_{B_{ \frac{1}{k}}}(\nabla u_k\cdot \nabla\psi)w_k\mathrm d x\right|+\left|\int_{B_{ \frac{1}{k}}} (\nabla u_k\cdot \nabla\psi)w\mathrm d x\right|+ \frac{1}{2}\epsilon<\epsilon.
		\end{split}
	\end{equation*}  
	This implies that 
	\begin{equation*}
		\int_\Omega (\nabla \widetilde u_0\cdot \nabla\psi)w\mathrm d x=\int_\Omega f\psi\mathrm d x. 
	\end{equation*}
	This proves that $\widetilde u_0$ is a solution of \eqref{08.15.1}. 
	
	Finally, by the uniqueness of the solution of the equation \eqref{08.15.1}, we get $\widetilde u_0=u_0$. 
	This complete the proof of the lemma. 
\end{proof}

Next, we transform the inhomogeneous problem 
\eqref{08.15.1} into a boundary value problem  to facilitate the subsequent proof of the three sphere inequality, and the proof of WUCP.

\begin{corollary}\label{08.19.C2}
	Let $u_0\in H_{w}^1(\Omega)$ be a solution of the equation \begin{equation}\label{08.19.1}
		\begin{cases}
			-\mathrm{div}(w\nabla u)=0, &\mbox{in }\Omega,\\
			u=g, &\mbox{on }\partial \Omega,
		\end{cases}
	\end{equation} 
	where $g\in H^{\frac{3}{2}}(\partial\Omega)$ is a given function.  Let $u_k\in H_{w_k}^1(\Omega)$ be a solution of the following equation 
	\begin{equation*}
		\begin{cases}
			-\mathrm{div}(w_k\nabla u_k)=0, &\mbox{in }\Omega,\\
			u_k=g, &\mbox{on }\partial\Omega. 
		\end{cases}
	\end{equation*}
	Then we have 
	\begin{equation*}
		\begin{split}  
			u_{k}
			&\rightarrow u_0 \mbox{ strongly in } L^2(\Omega). 
		\end{split}
	\end{equation*}
\end{corollary}

\begin{proof}
	We denote by $v\in H^2(\Omega)$ a solution of the following  equation
	\begin{equation*}
		\begin{cases} 
			-\Delta v+v=0, &\mbox{in }\Omega, \\
			v=g, &\mbox{on }\partial\Omega\,.  
		\end{cases}
	\end{equation*}
	Let $R_0>0$ with $B_{3R_0}\subseteq \Omega$  and take the cut-off function $\zeta\in C_0^\infty (\Omega)$ such that 
	\begin{equation*}
		0\leq \zeta\leq 1, \quad \zeta=0 \mbox{ on } B_{R_0},\quad \zeta=1 \mbox{ on }\Omega\setminus B_{2R_0}, \quad |\nabla \zeta|\leq  \frac{C}{R_0}, 
	\end{equation*}
	where $C>0$ is a generic constant. Set $v_0=\zeta v$. It is obvious that $\mathrm{div}(w\nabla v_0)\in L^2(\Omega; w)$ since $v_0=0$ on $B_{R_0}$, and $v_0=g$ on $\partial\Omega$. 
	
	Taking $\widetilde u_0=u_0-v_0$, then $\widetilde u_0\in H_{w,0}^1(\Omega)$ satisfies the following equation 
	\begin{equation*}
		\begin{cases}
			-\mathrm{div}(w\nabla \widetilde u_0)=\mathrm{div}(w \nabla v_0), &\mbox{in }\Omega, \\
			\widetilde u_0=0, &\mbox{on }\partial\Omega, 
		\end{cases}
	\end{equation*}
	By Lemma \ref{08.15.L5}, there exists a sequence $u_{k,0}\in H_{w_k,0}^1(\Omega)$ satisfies 
	\begin{equation*}
		\begin{cases} 
			-\mathrm{div}(w_k\nabla u_{k,0})=\mathrm{div}(w\nabla v_0), &\mbox{in }\Omega,\\
			u_{k,0}=0, &\mbox{on }\partial\Omega, 
		\end{cases}
	\end{equation*}
	and 
	\begin{equation*}
		\begin{split}
			u_{k,0}
			&\rightharpoonup \widetilde u_0 \mbox{ weakly in } H_{w, 0}^1(\Omega),\\
			u_{k,0}
			&\rightarrow \widetilde u_0 \mbox{ strongly in } L^2(\Omega). 
		\end{split}
	\end{equation*}
	Denote $u_k=u_{k,0}+v_0$.  Notice that $v_0=0$ on $B_{R_0}$ and $w=w_k$  for $k>  \frac{1}{R_0}$, then  $u_k$ is the solution of the following system 
	\begin{equation*}
		\begin{cases}
			-\mathrm{div}(w_k\nabla u_k)=0, &\mbox{in }\Omega, \\
			u_k=g, &\mbox{on }\partial\Omega, 
		\end{cases}
	\end{equation*}
	and 
	\begin{equation*}
		\begin{split} 
			u_{k}
			&\rightarrow u_0 \mbox{ strongly in } L^2(\Omega). 
		\end{split}
	\end{equation*}
	This complete the proof of Corollary \ref{08.19.C2}. 
\end{proof}

Before proving the three sphere inequality, we present some preliminary results. The proof of the degenerate three sphere inequality is more complex than that of the standard one, and the following corollary will play a crucial role in establishing the degenerate version.

\begin{corollary}\label{08.18.C1}
	Let $u_0, u_k\ (k\in\mathbb{N})$ be defined as in Lemma \ref{08.15.L5} or Corollary \ref{08.19.C2}. Then for any $\eta>0$ with $B_\eta\subseteq \Omega$, we have 
	\begin{equation}\label{co1}
		\int_{B_\eta}w_k u_k^2\mathrm d x\rightarrow \int_{B_\eta}w u_0^2\mathrm d x\mbox{ as } k\rightarrow\infty, 
	\end{equation}
	and 
	\begin{equation}\label{co2}
		R_k:= \frac{\alpha}{2k^2}\int_{B_ \frac{1}{k}}\left( \frac{3}{4}|y|^2+ \frac{1}{4k^2}\right)^{ \frac{\alpha}{2}-1}u_k^2\mathrm d y \stackrel{k\rightarrow\infty}{\longrightarrow} 0 . 
	\end{equation}
\end{corollary}

\begin{proof}
	Since $\eta > 0$, let $k \in \mathbb{N}$ be sufficiently large such that $ \frac{1}{k} <  \frac{1}{2}\eta$. Since $w = w_k$ in $B_\eta \setminus B_ \frac{1}{k}$, by Lemma \ref{08.15.L5} and the uniform continuity property, we see  that
	\begin{equation*}
		\begin{split}
			&\hspace{4.5mm}\left|\int_{B_\eta}w_ku_k^2\mathrm d x-\int_{B_\eta}wu_0^2\mathrm d x\right| \\
			&\leq \left|\int_{B_\eta}w_ku_k^2\mathrm d x-\int_{B_\eta}wu_k^2\mathrm d x\right|+\left|\int_{B_\eta}wu_k^2\mathrm d x-\int_{B_\eta}wu_0^2\mathrm d x\right|\\
			&\leq \max_{|x|\leq  \frac{1}{k}}\left[\left( \frac{3}{4}|x|^2+ \frac{1}{4k^2}\right)^{ \frac{\alpha}{2}}-|x|^\alpha\right]\|u_k\|_{L^2(B_\eta)}^2+\left|\int_{B_\eta}wu_k^2\mathrm d x-\int_{B_\eta}wu_0^2\mathrm d x\right|\\
			&\leq C \max_{|x|\leq  \frac{1}{k}}\left[\left( \frac{3}{4}|x|^2+ \frac{1}{4k^2}\right)^{ \frac{\alpha}{2}}-|x|^\alpha\right]+\left|\int_{B_\eta}wu_k^2\mathrm d x-\int_{B_\eta}wu_0^2\mathrm d x\right| \stackrel{k\rightarrow\infty}{\longrightarrow} 0 .  
		\end{split}
	\end{equation*}
	
	This proves the \eqref{co1}.   Now, we  are going to prove  \eqref{co2}. 
	
	Note that $ \frac{1}{4k^2}\leq  \frac{3}{4}|y|^2+ \frac{1}{4k^2}\leq  \frac{1}{k^2}$ in $B_ \frac{1}{k}$, we have 
	\begin{equation*}
		\frac{\alpha}{2}\int_{B_ \frac{1}{k}}w_ku_k^2\mathrm d x\leq R_k\leq 2\alpha\int_{B_ \frac{1}{k}}w_k u_k^2\mathrm d x, 
	\end{equation*}
	so, we only need to show that 
	\begin{equation*}
		\int_{B_ \frac{1}{k}}w_ku_k^2\mathrm d x \stackrel{k\rightarrow\infty}{\longrightarrow} 0 . 
	\end{equation*}
	Indeed, by \eqref{08.15.12}, we have 
	\begin{equation*}
		\begin{split}
			&\hspace{4.5mm}\int_{B_ \frac{1}{k}}w_ku_k^2\mathrm d x
			\leq \left|\int_{B_ \frac{1}{k}}w_ku_k^2\mathrm d x-\int_{B_ \frac{1}{k}}wu_0^2\mathrm d x\right|+\int_{B_ \frac{1}{k}}wu_0^2\mathrm d x\\
			&\leq \left|\int_{B_{ \frac{1}{k}}}(w_ku_k^2-w_ku_0^2)\mathrm d x\right|+\left|\int_{B_{ \frac{1}{k}}}(w_k-w)u_0^2\mathrm d x\right|+\int_{B_{ \frac{1}{k}}}wu_0^2\mathrm d x\\
			&\leq  \frac{1}{k^\alpha}\left|\int_{B_ \frac{1}{k}}u_k^2\mathrm d x-\int_{B_{ \frac{1}{k}}}u_0^2\mathrm d x\right|+2m^\alpha\int_{B_{ \frac{1}{k}}}u_0^2\mathrm d x+\int_{B_{ \frac{1}{k}}}wu_0^2\mathrm d x\stackrel{k\rightarrow\infty}{\longrightarrow} 0 . 
		\end{split}
	\end{equation*}
\end{proof}

\section{Three sphere  inequality}\label{S4}
We now proceed to prove the degenerate three sphere inequality, also using an approximation method. First, we introduce some notations.

Set  $0<\epsilon\ll r$, let
\begin{equation*}
	H(r)=\int_{B_r}w|v(y)|^2\mathrm d y,\quad D(r)=\int_{B_r}w |\nabla v(y)|^2(r^{2}-|y|^{2})\mathrm d y, 
\end{equation*}
and
\begin{equation*}
	H_\epsilon(r)=\int_{B_r}w_\epsilon|v(y)|^2\mathrm d y,\quad D_\epsilon(r)=\int_{B_r}w_\epsilon|\nabla v(y)|^2(r^{2}-|y|^{2})\mathrm d y, 
\end{equation*}
where $w_\epsilon$ defined in \eqref{we}. Let 
\begin{equation*}
	\Phi_\epsilon(r)=
	\begin{cases}
		\frac{D_\epsilon(r)}{H_\epsilon(r)}, & \mbox{if } H_\epsilon\neq 0, \\
		0, & \mbox{if } H_\epsilon=0. 
	\end{cases}
\end{equation*}

We first prove the following three sphere inequality for the uniformly elliptic operator with $\epsilon$,
which contains additional terms $ \frac{1}{2}\int_{r_1}^{r_2} \frac{r^\alpha R_\epsilon}{H_\epsilon(r)}\mathrm d r$ and $ \frac{1}{2}\int_{r_2}^{r_3} \frac{r^\alpha R_\epsilon}{H_\epsilon(r)}\mathrm d r$,
compared to the standard form of the three sphere inequality. 

\begin{lemma}\label{08.18.L1}
	Let  $R_0>0$ with $B_{2R_0}\subseteq \Omega$, and let $\epsilon\in (0,  \frac{1}{2}R_0)$ be small enough. Let $v$ be a solution of $ \mathrm{div}(w_\epsilon\nabla v)=0$ in $B_{R_0}$.  Then for any $0<2\epsilon<r_1<r_2<r_3<R_0$, we have 
	\begin{equation}\label{08.18.10}
		\begin{split}
		&\hspace{4.5mm}\frac{1}{r_1^{-\alpha}-r_2^{-\alpha}}\left(\log \frac{H_\epsilon(r_2)}{H_\epsilon(r_1)}+ \frac{1}{2}\int_{r_1}^{r_2} \frac{r^\alpha R_\epsilon}{H_\epsilon(r)}\mathrm d r\right)\\
		&\leq  \frac{1}{r_2^{-\alpha}-r_3^{-\alpha}}\left( \log \frac{H_\epsilon(r_3)}{H_\epsilon(r_2)}+ \frac{1}{2}\int_{r_2}^{r_3} \frac{r^\alpha R_\epsilon}{H_\epsilon(r)}\mathrm d r\right), 
				\end{split}
	\end{equation}
	where 
	\begin{equation}\label{08.18.2}
		R_\epsilon= \frac{\alpha}{2}\epsilon^2\int_{B_\epsilon}\left( \frac{3}{4}|y|^2+ \frac{1}{4}\epsilon^2\right)^{ \frac{\alpha}{2}-1}v^2\mathrm d y. 
	\end{equation}
\end{lemma}

\begin{proof} 
	It is obvious that $v\in H^2(B_{R_0})$ since $( \frac{\epsilon}{2})^\alpha\leq w_\epsilon$ on $\Omega$. i.e., $-\mathrm{div}(w_\epsilon\nabla v)=0$ in $B_{R_0}$ is indeed a uniformly elliptic equation. 
	We divide our proof into the following steps.

	\noindent {\it Step 1}. We compute $H_\epsilon'(r)$ and $D_\epsilon(r)$.   
	It is clear  that 
	\begin{equation}\label{08.17.3}
		H_\epsilon'(r)=\int_{\partial B_r}w_\epsilon|v(y)|^2\mathrm d\sigma(y). 
	\end{equation}  
	By 
	\begin{equation*}
		\int_{B_r}\mathrm{div}\big(w_\epsilon(\nabla v)v (r^{2}-|y|^{2})\big)\mathrm d y=\int_{\partial B_r}w_\epsilon(\nabla v\cdot \nu)v(r^{2}-|y|^{2})\mathrm d\sigma(y)=0, 
	\end{equation*}
	and 
	\begin{equation*}
		\begin{split}
			&\hspace{4.5mm}\int_{B_r}\mathrm{div}\big(w_\epsilon(\nabla v)v(r^{2}-|y|^{2})\big)\mathrm d y\\
			&=\int_{B_r}\mathrm{div}(w_\epsilon\nabla v)v(r^{2}-|y|^{2})\mathrm d y+\int_{B_r}w_\epsilon|\nabla v|^2(r^{2}-|y|^{2})\mathrm d y\\
			&\hspace{4.5mm}-2\int_{B_r} w_\epsilon v\nabla v\cdot y\mathrm d y  
		\end{split}
	\end{equation*}
	we get
	\begin{equation}\label{08.17.4}
		D_\epsilon(r)=2\int_{B_r}w_\epsilon v\nabla v\cdot y\mathrm d y
	\end{equation}
	by $ \mathrm{div}(w_\epsilon\nabla v)=0$. 
	
	\noindent {\it Step 2}. We compute $ \frac{H_\epsilon'(r)}{H_\epsilon(r)}$.   
	Set 
	\begin{equation*}
		G(y)=r^2-|y|^2, 
	\end{equation*}
	then 
	\begin{equation*}
		G|_{\partial B_r}=0, \quad \nabla G=-2y,\quad  \frac{\partial G}{\partial \nu}\bigg|_{\partial B_r}=-2r. 
	\end{equation*}
	
	We compute
	\begin{equation*}
		\begin{split}
			&\hspace{4.5mm}\int_{B_r}\mathrm{div}(w_\epsilon\nabla v^2)G\mathrm d y\\
			&=\int_{B_r}\mathrm{div}(w_\epsilon G\nabla v^2)\mathrm d y-\int_{B_r}w_\epsilon\nabla v^2\cdot \nabla G\mathrm d y\\
			&=\int_{\partial B_r}w_\epsilon G\nabla v^2\cdot \nu\mathrm d \sigma(y)-\int_{B_r}w_\epsilon\nabla v^2\cdot \nabla G\mathrm d y\\
			&=-\int_{B_r}w_\epsilon\nabla G\cdot \nabla v^2\mathrm d y=-\int_{B_r}\mathrm{div}(w_\epsilon v^2\nabla G)\mathrm d y+\int_{B_r}v^2\mathrm{div}(w_\epsilon\nabla G)\mathrm d y\\
			&=2r\int_{\partial B_r}w_\epsilon v^2\mathrm d \sigma(y)-2(N+\alpha)\int_{B_r}w_\epsilon v^2\mathrm d y\\
			&\hspace{4.5mm}+ \frac{\alpha}{2}\epsilon^2\int_{B_\epsilon}\left( \frac{3}{4}|y|^2+ \frac{1}{4}\epsilon^2\right)^{ \frac{\alpha}{2}-1}v^2\mathrm d y. 
		\end{split}
	\end{equation*}
	
	Since
	\begin{equation*}
		\begin{split}
		&\hspace{4.5mm}\int_{B_r}\mathrm{div}(w_\epsilon\nabla v^2)G\mathrm d y
			=2\int_{B_r}\mathrm{div}(w_\epsilon v\nabla v)G\mathrm d x\\
			&=2\int_{B_r}w_\epsilon|\nabla v|^2G\mathrm d x+2\int_{B_r}v\mathrm{div}(w_\epsilon\nabla v)G\mathrm d x=2D_\epsilon(r) 
		\end{split}
	\end{equation*}
	by \eqref{08.17.3}, we have 
	\begin{equation}\label{08.18.4}
		\begin{split}
			H_\epsilon'(r)= \frac{N+\alpha}{r}H_\epsilon(r)+ \frac{1}{r}D_\epsilon(r)- \frac{1}{2r}R_\epsilon, 
		\end{split}
	\end{equation}
	where 
	\begin{equation*}
		R_\epsilon= \frac{\alpha}{2}\epsilon^2\int_{B_\epsilon}\left( \frac{3}{4}|y|^2+ \frac{1}{4}\epsilon^2\right)^{ \frac{\alpha}{2}-1}v^2\mathrm d y. 
	\end{equation*}
	This implies 
	\begin{equation}\label{08.21.1}
		\frac{H_\epsilon'(r)}{H_\epsilon(r)}= \frac{N+\alpha}{r}+ \frac{1}{r} \frac{D_\epsilon(r)}{H_\epsilon(r)}- \frac{1}{2r} \frac{R_\epsilon}{H_\epsilon(r)}. 
	\end{equation}
	
	\noindent  {\it Step 3}. We compute $D_\epsilon'(r)$.  
	Now, 
	\begin{equation}\label{08.18.1}
		\begin{split}
			D_\epsilon'(r)=2r\int_{B_r}w_\epsilon|\nabla v(y)|^2\mathrm d y. 
		\end{split}
	\end{equation}
	
	On one hand, we have 
	\begin{equation*}
		\begin{split}
			&\hspace{4.5mm}\int_{B_r}\mathrm{div}\big[w_\epsilon|\nabla v(y)|^2(r^2-|y|^2)y\big]\mathrm d y\\
			&=\int_{\partial B_r}w_\epsilon|\nabla v(y)|^2\left(r^2-|y|^2\right)y\cdot\nu\mathrm d \sigma(y)=0. 
		\end{split}
	\end{equation*}
	On the other hand, we have 
	\begin{equation*}
		\begin{split}
			&\hspace{4.5mm}\int_{B_r}\mathrm{div}\big[w_\epsilon|\nabla v(y)|^2(r^2-|y|^2)y\big]\mathrm d y\\
			&=(N+\alpha)\int_{B_r}w_\epsilon|\nabla v(y)|^2(r^2-|y|^2)\mathrm d y\\
			&\hspace{4.5mm}- \frac{\alpha}{4}\epsilon^2\int_{B_\epsilon}\left( \frac{3}{4}|y|^2+ \frac{1}{4}\epsilon^2\right)^{ \frac{\alpha}{2}-1}|\nabla v(y)|^2(r^2-|y|^2)\mathrm d y\\
			&\hspace{4.5mm}-2\int_{B_r} w_\epsilon|y|^2|\nabla v(y)|^2\mathrm d y+\int_{B_r}w_\epsilon(y\cdot \nabla |\nabla v(y)|^2)(r^2-|y|^2)\mathrm d y.
		\end{split}
	\end{equation*}
	Now, by the equation $\nabla v\cdot \nabla (y\cdot \nabla v)=|\nabla v|^2+ \frac{1}{2}y\cdot \nabla |\nabla v|^2$ and $\mathrm{div}(w_\epsilon\nabla v)=0$ and $ \mathrm{div}(w_\epsilon\nabla v)=0$, we get
	\begin{equation*}
		\begin{split}
			&\hspace{4.5mm}\int_{B_r}w_\epsilon y\cdot \nabla |\nabla v(y)|^2(r^2-|y|^2)\mathrm d y\\
			&=-2\int_{B_r}w_\epsilon|\nabla v|^2(r^2-|y|^2)\mathrm d y+2\int_{B_r}w_\epsilon\nabla v\cdot \nabla (y\cdot \nabla v)(r^2-|y|^2)\mathrm d y\\
			&=-2D_\epsilon(r)+2\int_{B_r}\mathrm{div}\left[w_\epsilon(r^2-|y|^2)(y\cdot \nabla v)\nabla v\right]\mathrm d x\\
			&\hspace{4.5mm}-2\int_{B_r}(y\cdot \nabla v)\mathrm{div}\left[w_\epsilon(r^2-|y|^2)\nabla v\right]\mathrm d y\\
			&=-2D_\epsilon(r)-2\int_{B_r}(y\cdot \nabla v)\mathrm{div}(w_\epsilon\nabla v)(r^2-|y|^2)\mathrm d y+4\int_{B_r}w_\epsilon(y\cdot \nabla v)^2\mathrm d y\\
			&=-2D_\epsilon(r)+4\int_{B_r}w_\epsilon(y\cdot \nabla v)^2\mathrm d y. 
		\end{split}
	\end{equation*}
	Then 
	\begin{equation}\label{08.18.11}
		\begin{split}
			(N+\alpha-2)D_\epsilon(r)-2\int_{B_r}w_\epsilon|y|^2|\nabla v(y)|^2\mathrm d y+4\int_{B_r}w_\epsilon(y\cdot \nabla v)^2\mathrm d y-\widetilde R_\epsilon=0, 
		\end{split}
	\end{equation}
	where 
	\begin{equation*}
		\widetilde R_\epsilon= \frac{\alpha}{4}\epsilon^2\int_{B_\epsilon}\left( \frac{3}{4}|y|^2+ \frac{1}{4}\epsilon^2\right)^{ \frac{\alpha}{2}-1}|\nabla v(y)|^2(r^2-|y|^2)\mathrm d y. 
	\end{equation*}
	
	Since 
	\begin{equation*}
		D_\epsilon(r)=r^2\int_{B_r}w_\epsilon|\nabla v|^2\mathrm d y-\int_{B_r}w_\epsilon|y|^2|\nabla v|^2\mathrm d y, 
	\end{equation*}
	we obtain that 
	\begin{equation*}
		-\int_{B_r}w_\epsilon|y|^2|\nabla v|^2\mathrm d y=D_\epsilon(r)-r^2\int_{B_r}w_\epsilon|\nabla v|^2\mathrm d y, 
	\end{equation*}
	and 
	\begin{equation*}
		(N+\alpha)D_\epsilon(r)=2r^2\int_{B_r}w_\epsilon|\nabla v|^2\mathrm d y-4\int_{B_r}w_\epsilon(y\cdot \nabla v)^2\mathrm d y+\widetilde R_\epsilon. 
	\end{equation*}
	Hence, 
	\begin{equation}\label{08.18.3}
		D_\epsilon'(r)= \frac{N+\alpha}{r}D_\epsilon(r)+ \frac{4}{r}\int_{B_r}w_\epsilon(y\cdot \nabla v)^2\mathrm d y- \frac{1}{r}\widetilde R_\epsilon. 
	\end{equation}
	
	\noindent {\it Step 4}. We prove $r^\alpha\Phi_\epsilon(r)$ is a nondecreasing function for $r>0$.  
	Note that $R_\epsilon\geq 0, \widetilde R_\epsilon\geq 0$, we obtain 
	\begin{equation*}
		\begin{split} 
			&\hspace{4.5mm}H_\epsilon^2(r)\Phi_\epsilon'(r)
			=D_\epsilon'(r)H_\epsilon(r)-D_\epsilon(r)H_\epsilon'(r)\\
			&= \frac{4}{r}\left[\left(\int_{B_r}w_\epsilon(y\cdot \nabla v)^2\mathrm d x\right)\left(\int_{B_r}w_\epsilon v^2\mathrm d x\right)-\left(\int_{B_r}w_\epsilon vy\cdot \nabla v\mathrm d x\right)^2\right]\\
			&\hspace{4.5mm}- \frac{1}{r}\widetilde R_\epsilon H_\epsilon(r)+ \frac{1}{2r}D_\epsilon(r)R_\epsilon\\
			&\geq - \frac{1}{r}\widetilde R_\epsilon H_\epsilon(r). 
		\end{split}
	\end{equation*}
	Since $ \frac{\epsilon^2}{4}\leq  \frac{3}{4}|y|^2+ \frac{1}{4}\epsilon^2\leq \epsilon^2$, we have 
	\begin{equation*}
		\begin{split}
			\widetilde R_\epsilon
			&= \frac{\alpha}{4}\epsilon^2\int_{B_\epsilon}\left( \frac{3}{4}|y|^2+ \frac{1}{4}\epsilon^2\right)^{ \frac{\alpha}{2}-1}|\nabla v(y)|^2(r^2-|y|^2)\mathrm d y\\
			&\leq \alpha\int_{B_\epsilon}\left( \frac{3}{4}|y|^2+ \frac{1}{4}\epsilon^2\right)^{ \frac{\alpha}{2}}|\nabla v(y)|^2(r^2-|y|^2)\mathrm d y\leq \alpha D_\epsilon(r), 
		\end{split}
	\end{equation*}
	which implies that 
	\begin{equation*}
		\Phi_\epsilon'(r)\geq - \frac{\alpha}{r}\Phi_\epsilon(r). 
	\end{equation*}
	Hence 
	\begin{equation}\label{08.18.7}
		r^\alpha \Phi_\epsilon(r) \mbox{ is a nondecreasing function for $r>0$}. 
	\end{equation}
	
	\noindent {\it Step 5}.  Conclusion of the proof.  
	Again, by \eqref{08.21.1}, we have 
	\begin{equation*}
		\frac{\mathrm d}{\mathrm d t}\log H_\epsilon(r)= \frac{H_\epsilon'(r)}{H_\epsilon(r)}= \frac{1}{r}\left((N+\alpha)+\Phi(r)- \frac{1}{2} \frac{R_\epsilon}{H_\epsilon(r)}\right)\,. 
	\end{equation*}
	Then, for $0<r_1<r_2$, we have 
	\begin{equation*}
		\begin{split} 
			\log \frac{H_\epsilon(r_2)}{H_\epsilon(r_1)}
			&= \int_{r_1}^{r_2} \frac{1}{r}\left((N+\alpha)+\Phi_\epsilon(r)- \frac{1}{2} \frac{R_\epsilon}{H_\epsilon(r)}\right)\mathrm d r\\
			&=\int_{r_1}^{r_2} \frac{1}{r^{1+\alpha}}\left(r^\alpha(N+\alpha)+r^\alpha\Phi_\epsilon(r)- \frac{1}{2} \frac{r^\alpha R_\epsilon}{H_\epsilon(r)}\right)\mathrm d r, 
		\end{split}
	\end{equation*}
	and hence,  by \eqref{08.21.1}, we obtain 
	\begin{equation}\label{08.18.5}
		\begin{split}
			\log \frac{H_\epsilon(r_2)}{H_\epsilon(r_1)}+ \frac{1}{2}\int_{r_1}^{r_2} \frac{ R_\epsilon}{rH_\epsilon(r)}\mathrm d r
			&\leq \alpha^{-1}(r_1^{-\alpha}-r_2^{-\alpha})\big(r_2^\alpha (N+\alpha)+r_2^\alpha \Phi_\epsilon(r_2)\big). 
		\end{split}
	\end{equation}
	We note that the integral $\int_{r_1}^{r_2} \frac{ R_\epsilon}{rH_\epsilon(r)}\mathrm d x$ is meaningful since $R_\epsilon\leq 2\alpha H_\epsilon(r)$ for all $r>0$ and all $\epsilon>0$. 
	Now, for $r_2<r_3<R_0$, we have 
	\begin{equation*}
		\begin{split}
			\log \frac{H_\epsilon(r_3)}{H_\epsilon(r_2)}
			&=\int_{r_2}^{r_3} \frac{1}{r^{1+\alpha}}\left(r^\alpha(N+\alpha)+r^\alpha\Phi_\epsilon(r)- \frac{1}{2} \frac{r^\alpha R_\epsilon}{H_\epsilon(r)}\right)\mathrm d r, 
		\end{split}
	\end{equation*}
	and hence, by \eqref{08.18.7}, we obtain 
	\begin{equation}\label{08.18.6}
		\begin{split}
			\log \frac{H_\epsilon(r_3)}{H_\epsilon(r_2)}+ \frac{1}{2}\int_{r_2}^{r_3} \frac{ R_\epsilon}{rH_\epsilon(r)}\mathrm d r\geq \alpha^{-1}(r_2^{-\alpha}-r_3^{-\alpha})\big(r_2^\alpha (N+\alpha)+r_2^\alpha\Phi_\epsilon(r_2)\big). 
		\end{split}
	\end{equation}
	Combining \eqref{08.18.5} and \eqref{08.18.6},  for all $\epsilon>0$, we get 
	\begin{equation*}
		\begin{split}
		&\hspace{4.5mm}\frac{1}{r_1^{-\alpha}-r_2^{-\alpha}}\left(\log \frac{H_\epsilon(r_2)}{H_\epsilon(r_1)}+ \frac{1}{2}\int_{r_1}^{r_2} \frac{R_\epsilon}{rH_\epsilon(r)}\mathrm d r\right)\\
		&\leq  \frac{1}{r_2^{-\alpha}-r_3^{-\alpha}}\left( \log \frac{H_\epsilon(r_3)}{H_\epsilon(r_2)}+ \frac{1}{2}\int_{r_2}^{r_3} \frac{ R_\epsilon}{rH_\epsilon(r)}\mathrm d r\right). 
				\end{split}
	\end{equation*}
	This complete the proof of the lemma. 
\end{proof}

By a limiting argument, we obtain the following degenerate three sphere inequality.

\begin{theorem}\label{08.19.P1}
	Let $R_0>0$ with $B_{2R_0}\subseteq \Omega$. Let $u_0$ be a solution of $-\mathrm{div}(w\nabla u_0)=0$ in $B_{R_0}$. Then, for any  $0<r_1<r_2<r_3<R_0$, we have
	\begin{equation}\label{H1}
		\frac{1}{r_1^{-\alpha}-r_2^{-\alpha}}\log \frac{H(r_2)}{H(r_1)}\leq  \frac{1}{r_2^{-\alpha}-r_3^{-\alpha}} \log \frac{H(r_3)}{H(r_2)}. 
	\end{equation}
	Moreover, there exists $\mu\in (0,1)$, such that  
	\begin{equation*}\label{H2}
		H(r_2)\leq (H(r_1))^\mu(H(r_3))^{1-\mu}. 
	\end{equation*}
\end{theorem}

\begin{proof}
	Without loss of generality, we assume that $u_0\neq 0$ in $B_{R_0}$. By the standard regularity enhancement method for elliptic equations in \cite{Evans}, we have $u_0\in H^2(B_{R_0}\backslash B_{\frac{R_0}{2}})$ and $u_0\in H^{\frac{3}{2}}(\partial B_{R_0})$, and by Corollary \ref{08.19.C2}, there exists $u_k\in H_{w_k}^1(B_{R_0})\ (k\in\mathbb{N})$  satisfying 
	\begin{equation*}
		\begin{cases}
			-\mathrm{div}(w_k\nabla u_k)=0, &\mbox{in } B_{R_0}, \\
			u_k=u_0, &\mbox{on }\partial B_{R_0}, 
		\end{cases}
	\end{equation*}
	and 
	\begin{equation}\label{08.19.3}
		u_k\rightarrow u_0 \mbox{ strongly in }L^2(B_{R_0}). 
	\end{equation}
	
	By Lemma \ref{08.18.L1}, we get \eqref{08.18.10}. Replacing $\epsilon$ by $\epsilon= \frac{1}{k}$,   we see 
	\begin{equation*}
		\begin{split}
		&\hspace{4.5mm}\frac{1}{r_1^{-\alpha}-r_2^{-\alpha}}\left(\log \frac{H_k(r_2)}{H_k(r_1)}+ \frac{1}{2}\int_{r_1}^{r_2} \frac{ R_k}{rH_k(r)}\mathrm d r\right)\\
		&\leq  \frac{1}{r_2^{-\alpha}-r_3^{-\alpha}}\left( \log \frac{H_k(r_3)}{H_k(r_2)}+ \frac{1}{2}\int_{r_2}^{r_3} \frac{ R_k}{rH_k(r)}\mathrm d r\right), 
				\end{split}
	\end{equation*}
	where 
	\begin{equation*}
		H_k(r)=\int_{B_r}w_k u_k^2\mathrm d x,\quad R_k= \frac{\alpha}{2k^2}\int_{B_ \frac{1}{k}}\left( \frac{3}{4}|y|^2+ \frac{1}{4k^2}\right)^{ \frac{\alpha}{2}-1}u_k^2\mathrm d y. 
	\end{equation*}
	From \eqref{08.19.3} and Corollary \ref{08.18.C1}, letting $k\rightarrow\infty$, we get 
	\begin{equation*}
		\frac{1}{r_1^{-\alpha}-r_2^{-\alpha}}\log \frac{H(r_2)}{H(r_1)}\leq  \frac{1}{r_2^{-\alpha}-r_3^{-\alpha}} \log \frac{H(r_3)}{H(r_2)}. 
	\end{equation*}
	This proves \eqref{H1}. 
	
	Finally, taking 
	\begin{equation*}
		\mu= \frac{r_1^{-\alpha}-r_2^{-\alpha}}{r_1^{-\alpha}-r_3^{-\alpha}}=\left( \frac{r_3}{r_2}\right)^\alpha  \frac{\left( \frac{r_2}{r_1}\right)^\alpha-1}{\left( \frac{r_3}{r_1}\right)^\alpha-1},  1-\mu= \frac{r_2^{-\alpha}-r_3^{-\alpha}}{r_1^{-\alpha}-r_3^{-\alpha}}=\left( \frac{r_1}{r_2}\right)^\alpha  \frac{1-\left( \frac{r_2}{r_3}\right)^\alpha}{1-\left( \frac{r_1}{r_3}\right)^\alpha}, 
	\end{equation*}
	yields 
	\begin{equation*}
		H(r_2)\leq (H(r_1))^\mu (H(r_3))^{1-\mu}. 
	\end{equation*}
	This complete the proof of theorem. 
\end{proof}
Below, we present the most standard and commonly used form of the degenerate three sphere inequality.
\begin{corollary}\label{08.19.C1}
	Assume the conditions in Theorem \ref{08.19.P1} hold. Then 
	\begin{equation*}
		\int_{B_r}v^2w\mathrm d x\leq \left(\int_{B_ \frac{r}{2}}v^2w\mathrm d x\right)^\mu\left(\int_{B_{2r}}v^2w\mathrm d y\right)^{1-\mu}
	\end{equation*}
	for $0< r< \frac{R_0}{2}$ with $\mu=\frac{4^\alpha-2^\alpha}{4^\alpha-1}$.
\end{corollary}

\begin{proof}
	Taking $r_1= \frac{r}{2}, r_2=r, r_3=2r$ in   Theorem \ref{08.19.P1} produces  the desired conclusion. 
\end{proof}
We have already obtained the three sphere inequality  at the degenerate point $0$. To derive an estimate over the entire domain, we will now present the three sphere inequality  at the non-degenerate point.
\begin{lemma}\label{05.19.L1}
	Let $\Gamma$ be a non-empty open subset of $\partial\Omega$. Let $r_0, r_1, r_2, r_3$ be four real numbers such that $0<r_1<r_0<r_2<r_3< \frac{R_0}{8}$. Suppose that $y_0\in \Omega, |y_0|>r_0$ satisfies the following three conditions: 
	
	{\rm i)} $B(y_0, r)\cap \Omega$ is star-shaped with respect to $y_0$ for all $r\in (0,  \frac{R_0}{4})$, 
	
	{\rm ii)} $B(y_0, r)\subseteq \Omega$ for all $r\in (0, r_0)$,
	
	{\rm iii)} $B(y_0, r)\cap \partial \Omega\subseteq \Gamma$ for all $r\in (r_0,  \frac{R_0}{2})$. 
	
	\noindent If $u\in H^2(\Omega)$ is a solution to  $ \mathrm{div}(w\nabla u)=0$ in $\Omega\setminus B_{2r_0}$ and $u=0$ on $\Gamma$, then there exists $\mu\in (0,1)$ such that 
	\begin{equation*}
		\int_{B(y_0, r_2)\cap D}v^2\mathrm d x\leq C\left(\int_{B(y_0, r_1)}v^2\mathrm d x\right)^\mu \left(\int_{B(y_0, r_3)\cap \Omega}v^2\mathrm d x\right)^{1-\mu}, 
	\end{equation*}
	and $C>0$ is a constant that only depends on $r_0, R_0$ and $N$. 
\end{lemma}

\begin{proof}
	Set $\widehat\Omega=\Omega\setminus\overline B(0,2r_0)$. We note that $\mathcal{A} v=0$ on $\widehat \Omega$ is an uniformly elliptic equation since
	\begin{equation*}
		2^\alpha r_0^\alpha|\xi|^2\leq \sum_{i,j=1}^N|x|^\alpha \xi_i\xi_j\leq \left(\sup_{x\in\Omega}|x|^\alpha \right)|\xi|^2, \ \forall \xi=(\xi_1,\cdots, \xi_N)\in\mathbb{R}^N
	\end{equation*}
	and $\Omega$ is a bounded domain. 
	Similar to \cite{Adolfsson,Kukavica1, Kukavica2,Phung, Vessella}, we obtain Lemma \ref{05.19.L1}. 
\end{proof}		

%
%

In the proof of WUCP, we consider two cases: when the degenerate point lies inside or outside $\omega$ ($0 \in \omega$).
To deal with the later one, we present a result analogous to the Schauder estimate.
Specifically, we estimate the integral over the small ball containing the origin by the integral over an annular region surrounding the ball.
In such a way, we control the integral in the degenerate region by the integral in the non-degenerate region. Similar Schauder estimates in the context of degenerate equations can be found in \cite{Gabriele,MR1475774,MR4207950}.
\begin{theorem}\label{08.19.T2}
	Let $R_0>0$ with $B_{2R_0}\subseteq \Omega$. Let $u$ be a solution to $ \mathrm{div}(w\nabla u)=0$ in $B_{R_0}$. Then there exists $C>0$ that is independent of $r\ (r<R_0)$ and $u$, such that 
	\begin{equation*}
		\begin{split}
			\int_{B_ \frac{r}{2}}u^2w\mathrm d x\leq  \frac{C}{r^2}\int_{B_r\setminus B_{ \frac{3}{4}r}}u^2w\mathrm d x. 
		\end{split}
	\end{equation*}
\end{theorem}

\begin{proof}
	Let $\zeta\in C_0^\infty(\mathbb{R}^N)$ be a cut-off function satisfying 
	\begin{equation*}
		\zeta=1 \mbox{ on } B_{ \frac{3}{4}r}, \mbox{ and } \zeta=0 \mbox{ on } \mathbb{R}^N-B_r, \mbox{ and } |\nabla \zeta|\leq  \frac{C}{r} \mbox{ on }B_r\setminus B_{ \frac{3}{4}r}, 
	\end{equation*}
	where $C>0$ is a generic constant. 
	
	Using $\zeta^2u$ as the test function, we have 
	\begin{equation*}
		0=\int_{B_{R_0}}\nabla u\cdot \nabla(\zeta^2u)w\mathrm d x. 
	\end{equation*}
	This implies that 
	\begin{equation*}
		\begin{split}
			&\hspace{4.5mm}\int_{B_{R_0}}\zeta^2|\nabla u|^2w\mathrm d x=-2\int_{B_{R_0}}\zeta u(\nabla\zeta \cdot\nabla u) w\mathrm d x\\
			&\leq  \frac{1}{2}\int_{B_{R_0}}\zeta^2|\nabla u|^2w\mathrm d x+4\int_{B_{R_0}}u^2|\nabla \zeta|^2w\mathrm d x
		\end{split}
	\end{equation*}
	by Cauchy inequality, i.e., 
	\begin{equation*}
		\int_{B_{R_0}}\zeta^2|\nabla u|^2w\mathrm d x\leq 8\int_{B_{R_0}}u^2|\nabla \zeta|^2w\mathrm d x. 
	\end{equation*}
	From which, we obtain that 
	\begin{equation*}
		\begin{split}
			\int_{B_{R_0}}|\nabla (\zeta u)|^2w\mathrm d x
			&\leq 2\int_{B_{R_0}}|\nabla \zeta|^2u^2w\mathrm d x+2\int_{B_{R_0}}\zeta^2|\nabla u|^2w\mathrm d x\\
			&\leq 18\int_{B_{R_0}}u^2|\nabla \zeta|^2w\mathrm d x. 
		\end{split}
	\end{equation*}
	Note that $\zeta u\in H_{w,0}^1(\Omega)$, by \eqref{hardy62} in Lemma \ref{08.15.L1}, we obtain 
	\begin{equation*}
		\int_{B_{R_0}}(\zeta u)^2w\mathrm d x\leq C\int_{B_{R_0}}u^2|\nabla \zeta|^2w\mathrm d x. 
	\end{equation*}
	This implies that 
	\begin{equation*}
		\int_{B_ \frac{r}{2}}u^2w\mathrm d x\leq  \frac{C}{r^2}\int_{B_r\setminus B_{ \frac{3}{4}r}}u^2w\mathrm d x
	\end{equation*}
	according to the definition of $\zeta$. 
\end{proof}

Next, we  provide the proof of WUCP, which is characterized by  the following two equivalent theorems.
\begin{theorem}\label{3-2.C8}
Let $\Gamma$ be a non-empty open subset of $\partial\Omega$  and let $\omega$ be a non-empty open subset of $\Omega$. Then, for each $D\subseteq \Omega$ satisfying  $\partial D\cap \partial \Omega\subset\subset \Gamma$ and $\overline D\setminus(\Gamma\cap \partial D)\subseteq \Omega$,  there exists $\mu\in (0,1)$, such that for any solution $v\in H_{w}^1(\Omega)$ of \eqref{08.15.1} with $v=0$ on $\Gamma$, we have 
\begin{equation}\label{05.19.4}
	\int_{D} v^2w\mathrm d y\leq C\left( \frac{1}{\epsilon}\right)^{ \frac{1-\mu}{\mu}}\int_\omega v^2w\mathrm d y+\epsilon\int_\Omega v^2w\mathrm d y
\end{equation}
for any $\epsilon>0$, where $C>0$ is a constant   independent of $u$.  
\end{theorem}

\begin{proof} 
We divide the proof into the following steps.

\noindent {\it Step 1}. 
There are two cases that we should consider: one is $0\in \omega$, the other is $0\notin \omega$.  
In what follows, we denote $k\in\mathbb{N}$ an arbitrary integer. 

{\it Case 1}. Assume $0\in \omega$.  
We choose $r_0>0$ such that $r_0$ satisfies the conditions of Lemma \ref{05.19.L1} and $B(0,r_0)\subseteq \omega$. Since $\Omega$ is connected, then there exists a compact set $K\subseteq D$, such that $B(q,r_0)\subseteq D$ for all $q\in K$, and $D\subseteq \bigcup_{q\in K}B(q,2r_0)$ and $B(q,2r_0)\cap \partial \Omega \subseteq \Gamma$ for all $q\in K$. Hence, 
there exists a sequence of balls $\{B(q_j,r_0)\}_{j=0,1,\cdots, k}$, such that the following conditions hold
\begin{equation}\label{subset2}
	\begin{split}
		B(q_{j+1},r_0)\subseteq B(q_j,2r_0) \mbox{ for all } j=0,1,\cdots, k-1,  \quad \mbox{and } q_0=0, q_k=q.
	\end{split}
\end{equation}
Note that $w\geq r_0^\alpha$ on $\Omega\setminus B_{r_0}$, then there exists $C>0$ that independents $v$, such that 
\begin{equation*}
	\begin{split}
		&\hspace{4.5mm}\int_{B(q_k,r_0)}v^2w\mathrm d y\ \left(\mbox{or},  \int_{B(q_k,2r_0)\cap D}v^2w\mathrm d y\right)\\ 
		&\leq C\left(\int_{B(q_k,r_0)}v^2w\mathrm d y\right)^{\mu_1}\left(\int_\Omega v^2w\mathrm d y\right)^{1-\mu_1}\\
		&\leq C\left(\int_{B(q_{k-1},2r_0)}v^2w\mathrm d y\right)^{\mu_1}\left(\int_\Omega v^2w\mathrm d y\right)^{1-\mu_1},		
	\end{split}
\end{equation*}
where $\mu_1$ is the exponent in Lemma \ref{05.19.L1} and we used Lemma \ref{05.19.L1} in the first inequality and \eqref{subset2} in the second. Then, we apply Lemma \ref{05.19.L1} again to  $\int_{B(q_{k-1},2r_0)}v^2w\mathrm d y$, which implies
\begin{equation*}
	\begin{split}
		\int_{B(q_k,r_0)}v^2w\mathrm d y
		&\leq C\left(\int_{B(q_{k-1},2r_0)}v^2w\mathrm d y\right)^{\mu_1}\left(\int_\Omega v^2w\mathrm d y\right)^{1-\mu_1}\\ 
		&\leq C\left(\int_{B(q_{k-1}, r_0)}v^2\mathrm d y\right)^{\mu_1^2}\left(\int_{\Omega}v^2w\mathrm d y\right)^{1-\mu_1^2}\\
		&\leq C\left(\int_{B(q_{k-2}, 2r_0)}v^2\mathrm d y\right)^{\mu_1^2}\left(\int_{ \Omega}v^2w\mathrm d y\right)^{1-\mu_1^2},
	\end{split}
\end{equation*}

Repeating the use of Lemma \ref{05.19.L1} and \eqref{subset2} in this way, we can deduce that
\begin{equation*}
	\begin{split}
		\int_{B(q_k,r_0)}v^2w\mathrm d y
		&\leq C\left(\int_{B(q_{k-2}, 2r_0)}v^2\mathrm d y\right)^{\mu_1^2}\left(\int_{ \Omega}v^2w\mathrm d y\right)^{1-\mu_1^2}\\
		&\leq \cdots\\ 
		&\leq C\left(\int_{B(q_1,r_0)}v^2w\mathrm d y\right)^{\mu_1^{k}}\left(\int_\Omega v^2w\mathrm d y\right)^{1-\mu_1^{k}}\\
		&\leq C\left(\int_{B(q_0,2r_0)}v^2w\mathrm d y\right)^{\mu_1^k}\left(\int_\Omega v^2w\mathrm d y\right)^{1-\mu_1^k}.\\ 
		\end{split}
\end{equation*}
Since $q_0=0$, by applying Corollary \ref{08.19.C1} to  $\int_{B(q_0,2r_0)}v^2w\mathrm d y$, we  conclude that
\begin{equation*}
	\begin{split}
		\int_{B(q_k,r_0)}v^2w\mathrm d y
		&\leq C\left(\int_{B(q_0,2r_0)}v^2w\mathrm d y\right)^{\mu_1^k}\left(\int_\Omega v^2w\mathrm d y\right)^{1-\mu_1^k}.\\ 
				&\leq C\left(\int_{B(0,r_0)}v^2w\mathrm d y\right)^{\mu_1^k\mu_2}\left(\int_\Omega v^2w\mathrm d y\right)^{1-\mu_1^k\mu_2}\,,  
	\end{split}
\end{equation*}

where  $\mu_2$ is the exponent in Corollary \ref{08.19.C1}. 

{\it Case 2}. Assume $0\notin \overline\omega$. 
We choose $r_0>0$ and $q_0\in \omega$ such that $r_0$ satisfies the conditions of Lemma \ref{05.19.L1} and $B(q_0,2r_0)\subseteq \omega$ and $B(q_0,r_0)\cap B(0,r_0)=\emptyset$. Note that in contrast to the choice made in Case 1, we now take $q_0 \neq 0$. Choosing $q_1,\cdots, q_k$ and let $q_k =0$ such that $|q_{j}-q_{j-1}|<r_0$ for $j=1,\cdots, k$,  and $B(q_j, 2r_0)\subseteq \Omega$ for $j=0,1,\cdots, k$. 
Then, for $q_k=0$, there exists $C>0$ that is  independent of $v$, such that 
\begin{equation}\label{zerok}
	\begin{split}
		\int_{B(q_k,2r_0)}v^2w\mathrm d y\leq  \frac{C}{r_0^{2}}\int_{A_{3r_0,  \frac{5}{2}r_0}}v^2w\mathrm d y
	\end{split}
\end{equation}
by Theorem \ref{08.19.T2} and $r_0< \frac{R_0}{4}$, where $A_{3r_0,  \frac{5}{2}r_0}=\{x\in\mathbb{R}^N\colon  \frac{5}{2}r_0<|x|<{3r_0}\}$ is an annulus. Note that 
\begin{equation*}
	f(q)=\int_{B(q, r_0)\cap A_{3r_0, \frac{5}{2}r_0}}v^2w\mathrm d y, \ \forall q\in \partial B\left(0,  \frac{5}{2}r_0\right)
\end{equation*}
is a continuous function, then there exists $q_{k-1}\in \partial B(0, \frac{5}{2}r_0)$ such that  
\begin{equation*}
	f(q_{k-1})=\max_{q\in \partial B(0, \frac{5}{2}r_0)} f(q). 
\end{equation*}
Hence, there exists  a constant $C\in \mathbb{Z}^+$ (the constant $C$ depends only on $N$  and $r_0$), such that $A_{3r_0, \frac{5}{2}r_0}$ can be covered by $C$ numbers $B(q,r_0)$ with $q\in \partial B(0, \frac{5}{2}r_0)$. Moreover, 
\begin{equation}\label{05.19.1}
	\int_{A_{3r_0,  \frac{5}{2}r_0}}v^2w\mathrm d y\leq C\int_{B(q_{k-1},r_0)}v^2w\mathrm d y.
\end{equation} 
Since $\Omega$ is connected, then there exists a compact set $K\subseteq D\setminus B(0,2r_0)$, such that $B(q,r_0)\subseteq D\setminus B(0,r_0)$ for all $q\in K$, and $D\setminus B(0,2r_0)\subseteq \bigcup_{q\in K}B(q,2r_0)$ and $B(q,2r_0)\cap \partial \Omega\subseteq\Gamma$ for all $q\in K$. Hence, there exists a sequence of balls $\{B(q_j,r_0)\}_{j=0,1,\cdots, k-1}$, such that the following conditions hold
\begin{equation*}
	\begin{split}
		B(q_{j+1},r_0)\subseteq B(q_j,2r_0) \mbox{ for all } j=0,1,\cdots, k-2.
	\end{split}
\end{equation*}
Now, we use Lemma \ref{05.19.L1} by $k-2$ times to obtain 
\begin{equation}\label{05.19.2}
	\begin{split}
		\int_{B(q_{k-1},r_0)}v^2w\mathrm d y
		&\leq \int_{B(q_{k-2},2r_0)}v^2w\mathrm d y\\
		&\leq C\left(\int_{B(q_{k-3}, r_0)}v^2w\mathrm d y\right)^{\mu_1}\left(\int_{\Omega}v^2w\mathrm d y\right)^{1-\mu_1}\\
		&\leq \cdots\leq C\left(\int_{B(q_0,r_0)}v^2w\mathrm d y\right)^{\mu_1^{k-2}}\left(\int_\Omega v^2w\mathrm d y\right)^{1-\mu_1^{k-2}}\,, 
	\end{split}
\end{equation}
where $\mu_1$ is the exponent in Lemma \ref{05.19.L1}. Finally, together \eqref{zerok}, \eqref{05.19.1} and \eqref{05.19.2} we have 
\begin{equation}\label{05.19.3}
	\begin{split}
		\int_{B(0,2 r_0)}v^2w\mathrm d y\leq C\left(\int_{B(q_0,r_0)}v^2w\mathrm d y\right)^{\mu}\left(\int_\Omega v^2w\mathrm d y\right)^{1-\mu}, 
	\end{split}
\end{equation}
where the constant $C>0$ is independent of $v$ but depends on $r_0$, and $\mu=\mu_1^{k-2}$. 

For  each point $q_m\in K\subseteq D\setminus B(0,2r_0)$, using  Lemma \ref{05.19.L1} by $m$ times, we obtain 
\begin{equation*}
	\begin{split}
	&\hspace{4.5mm}\int_{B(q_m,r_0)}v^2w\mathrm d y \left(\mbox{or},  \int_{B(q_m,2r_0)\cap D}v^2w\mathrm d y\right)\\
	 &\leq C\left(\int_{B(q_0,r_0)} v^2w\mathrm d y\right)^{\mu_1^{m}}\left(\int_\Omega v^2\mathrm d y\right)^{1-\mu_1^{m}},
	 	\end{split}
\end{equation*}
where $\mu_1$ is defined in Lemma \ref{05.19.L1}.

\noindent {\it Step 2}. By Case 1 and Case 2 in Step 1, using the case that $K$ is compact and by finite covering theorem, we obtain that 
\begin{equation}\label{wucp1}
	\int_{D} v^2w\mathrm d y\leq C\left(\int_\omega v^2w\mathrm d y\right)^\mu\left(\int_\Omega v^2w\mathrm d y\right)^{1-\mu}, 
\end{equation}
where $C>0$ is a constant that is independent of $v$. 

\noindent {\it Step 3}. The proof of \eqref{05.19.4} is standard. We denote 
\begin{equation*}
	A=\int_{D} v^2w\mathrm d y\neq 0, \quad B=\int_\omega v^2w\mathrm d y, \quad E=\int_\Omega v^2w\mathrm d y. 
\end{equation*}
Then, by Step 2, there exists $C>0$ and $\mu\in (0,1)$ such that $A\leq CB^\mu E^{1-\mu}$,    i.e., 
\begin{equation*}
	A\leq C^ \frac{1}{\mu}B\left( \frac{E}{A}\right)^ \frac{1-\mu}{\mu}. 
\end{equation*}
Now, if $ \frac{E}{A}\leq  \frac{1}{\epsilon}$, then $A\leq \epsilon E$. This implies \eqref{05.19.4}. This complete the proof of Theorem \ref{3-2.C8}. 
\end{proof}
Lastly, we present an equivalent result to Theorem \ref{3-2.C8}.
\begin{theorem}\label{3-2.T7}
Let $\Gamma$ be a non-empty open subset of $\partial\Omega$  and let $\omega$ be a non-empty open subset of $\Omega$. Then, for each $D\subseteq \Omega$ satisfying $\partial D\cap \partial \Omega\subset\subset \Gamma$ and $\overline D\setminus(\Gamma\cap \partial D)\subseteq \Omega$,  there exists  $\mu\in (0,1)$, such that for any solution $u\in H_{w}^1(\Omega)$ of \eqref{08.15.1} with $u=0$ on $\Gamma$, we have 
\begin{equation}\label{wucp2}
	\int_{D} u^2w\mathrm d y\leq C\left(\int_\omega u^2w\mathrm d y\right)^\mu\left(\int_\Omega u^2w\mathrm d y\right)^{1-\mu},
\end{equation}
where $C>0$ is a constant  independent of $u$. 
\end{theorem}

\begin{proof}
Assume \eqref{wucp2} is true, we just need to follow the {\it Step 3} in Theorem \ref{3-2.C8} to derive \eqref{05.19.4}.

Conversely, assume \eqref{05.19.4} is true, we denote
\begin{equation*}
	A=\int_{D} v^2w\mathrm d y\neq 0, \quad B=\int_\omega v^2w\mathrm d y, \quad E=\int_\Omega v^2w\mathrm d y, 
\end{equation*}
choose $\epsilon= \frac{1}{2}  \frac{A}{E}$, then $A\leq 2C B^\mu E^{1-\mu}$, and we obtain \eqref{wucp2}.
\end{proof}
Finally, we provide  WUCP for the degenerate  elliptic operator.

\begin{theorem}\label{05.21.C1}
The degenerate elliptic operator $-\mathrm{div}(w\nabla\cdot)$ on $\Omega$ satisfies the WUCP. 
\end{theorem}

\begin{proof}
By Theorem \ref{3-2.C8} and Theorem \ref{3-2.T7}, we can easily deduce $u \equiv 0$ in $\Omega$ from the assumption that $u = 0$ in $\omega$, which is precisely the WUCP.
\end{proof}

\vspace{3mm}

\noindent{\bf{Acknowledgment}}

\vspace{2mm}

We would like to express our gratitude to Dr. Yubiao Zhang from Tianjin University for providing valuable suggestions for this work.

\vspace{3mm}

\noindent{\bf{Declarations}}
The authors have not disclosed any competing interests and data availability.

\vspace{2mm}


\bibliographystyle{amsplain}
\bibliography{ref20250423}

\providecommand{\bysame}{\leavevmode\hbox to3em{\hrulefill}\thinspace}
\providecommand{\MR}{\relax\ifhmode\unskip\space\fi MR }
\providecommand{\MRhref}[2]{%
  \href{http://www.ams.org/mathscinet-getitem?mr=#1}{#2}
}
\providecommand{\href}[2]{#2}
\begin{thebibliography}{10}

\bibitem{Adolfsson}
Vilhelm Adolfsson and Luis Escauriaza, \emph{{$C^{1,\alpha}$} domains and
  unique continuation at the boundary}, Comm. Pure Appl. Math. \textbf{50}
  (1997), no.~10, 935--969. \MR{1466583}

\bibitem{Bakri}
Laurent Bakri, \emph{Carleman estimates for the {S}chr\"odinger operator.
  {A}pplications to quantitative uniqueness}, Comm. Partial Differential
  Equations \textbf{38} (2013), no.~1, 69--91. \MR{3005547}

\bibitem{Banerjee}
Agnid Banerjee, Nicola Garofalo, and Ramesh Manna, \emph{Carleman estimates for
  {B}aouendi-{G}rushin operators with applications to quantitative uniqueness
  and strong unique continuation}, Appl. Anal. \textbf{101} (2022), no.~10,
  3667--3688. \MR{4446597}

\bibitem{banerjee2020strong}
Agnid Banerjee and Arka Mallick, \emph{On the strong unique continuation
  property of a degenerate elliptic operator with {H}ardy-type potential}, Ann.
  Mat. Pura Appl. (4) \textbf{199} (2020), no.~1, 1--21. \MR{4065106}

\bibitem{MR2106129}
P.~Cannarsa, P.~Martinez, and J.~Vancostenoble, \emph{Null controllability of
  degenerate heat equations}, Adv. Differential Equations \textbf{10} (2005),
  no.~2, 153--190. \MR{2106129}

\bibitem{cannarsa2016}
\bysame, \emph{Global {C}arleman estimates for degenerate parabolic operators
  with applications}, Mem. Amer. Math. Soc. \textbf{239} (2016), no.~1133,
  ix+209. \MR{3430764}

\bibitem{Carleman}
T.~Carleman, \emph{Sur un probl\`eme d'unicit\'e{} pur les syst\`emes
  d'\'equations aux d\'eriv\'ees partielles \`a{} deux variables
  ind\'ependantes}, Ark. Mat. Astr. Fys. \textbf{26} (1939), no.~17, 9.
  \MR{334}

\bibitem{MR1794994}
Florin Catrina and Zhi-Qiang Wang, \emph{On the {C}affarelli-{K}ohn-{N}irenberg
  inequalities: sharp constants, existence (and nonexistence), and symmetry of
  extremal functions}, Comm. Pure Appl. Math. \textbf{54} (2001), no.~2,
  229--258. \MR{1794994}

\bibitem{Cavalheiro}
Albo~Carlos Cavalheiro, \emph{An approximation theorem for solutions of
  degenerate elliptic equations}, Proc. Edinb. Math. Soc. (2) \textbf{45}
  (2002), no.~2, 363--389. \MR{1912646}

\bibitem{Cavalheiro1}
\bysame, \emph{An approximation theorem for solutions of degenerate semilinear
  elliptic equations}, Commun. Math. \textbf{25} (2017), no.~1, 21--34.
  \MR{3667074}

\bibitem{Evans}
Lawrence~C. Evans, \emph{Partial differential equations}, second ed., Graduate
  Studies in Mathematics, vol.~19, American Mathematical Society, Providence,
  RI, 2010. \MR{2597943}

\bibitem{Fabes}
Eugene~B. Fabes, Carlos~E. Kenig, and Raul~P. Serapioni, \emph{The local
  regularity of solutions of degenerate elliptic equations}, Comm. Partial
  Differential Equations \textbf{7} (1982), no.~1, 77--116. \MR{643158}

\bibitem{Gabriele}
Gabriele Fioravanti, \emph{The dirichlet problem on lower dimensional
  boundaries: Schauder estimates via perforated domains}, preprint.

\bibitem{Franchi}
Bruno Franchi, Cristian~E. Guti\'errez, and Richard~L. Wheeden, \emph{Weighted
  {S}obolev-{P}oincar\'e{} inequalities for {G}rushin type operators}, Comm.
  Partial Differential Equations \textbf{19} (1994), no.~3-4, 523--604.
  \MR{1265808}

\bibitem{GC}
Jos\'e{} Garc\'ia-Cuerva and Jos\'e{}~L. Rubio~de Francia, \emph{Weighted norm
  inequalities and related topics}, North-Holland Mathematics Studies, vol.
  116, North-Holland Publishing Co., Amsterdam, 1985, Notas de Matem\'atica,
  104. [Mathematical Notes]. \MR{807149}

\bibitem{Garofalo3}
N.~Garofalo and Z.~Shen, \emph{Carleman estimates for a subelliptic operator
  and unique continuation}, Ann. Inst. Fourier (Grenoble) \textbf{44} (1994),
  no.~1, 129--166. \MR{1262883}

\bibitem{Garofalo4}
Nicola Garofalo, \emph{Unique continuation for a class of elliptic operators
  which degenerate on a manifold of arbitrary codimension}, J. Differential
  Equations \textbf{104} (1993), no.~1, 117--146. \MR{1224123}

\bibitem{Garofalo}
Nicola Garofalo and Fang-Hua Lin, \emph{Monotonicity properties of variational
  integrals, {$A_p$} weights and unique continuation}, Indiana Univ. Math. J.
  \textbf{35} (1986), no.~2, 245--268. \MR{833393}

\bibitem{Garofalo1987}
\bysame, \emph{Unique continuation for elliptic operators: a
  geometric-variational approach}, Comm. Pure Appl. Math. \textbf{40} (1987),
  no.~3, 347--366. \MR{882069}

\bibitem{Garofalo1}
Nicola Garofalo and Dimiter Vassilev, \emph{Strong unique continuation
  properties of generalized {B}aouendi-{G}rushin operators}, Comm. Partial
  Differential Equations \textbf{32} (2007), no.~4-6, 643--663. \MR{2334826}

\bibitem{Wu1}
B.~Guo, W.~Wu, and D.~Yang, \emph{Quantitative unique continuation for a class
  of degenerate elliptic equations with weakened boundary conditions},
  preprint.

\bibitem{Hormander}
Lars H\"ormander, \emph{On the uniqueness of the {C}auchy problem}, Math.
  Scand. \textbf{6} (1958), 213--225. \MR{104924}

\bibitem{Kenig}
Carlos~E. Kenig, \emph{Carleman estimates, uniform {S}obolev inequalities for
  second-order differential operators, and unique continuation theorems},
  Proceedings of the {I}nternational {C}ongress of {M}athematicians, {V}ol. 1,
  2 ({B}erkeley, {C}alif., 1986), Amer. Math. Soc., Providence, RI, 1987,
  pp.~948--960. \MR{934297}

\bibitem{Kukavica1}
Igor Kukavica, \emph{Level sets for the stationary solutions of the
  {G}inzburg-{L}andau equation}, Calc. Var. Partial Differential Equations
  \textbf{5} (1997), no.~6, 511--521. \MR{1473306}

\bibitem{Kukavica}
\bysame, \emph{Quantitative uniqueness for second-order elliptic operators},
  Duke Math. J. \textbf{91} (1998), no.~2, 225--240. \MR{1600578}

\bibitem{Kukavica2}
Igor Kukavica and Kaj Nystr\"om, \emph{Unique continuation on the boundary for
  {D}ini domains}, Proc. Amer. Math. Soc. \textbf{126} (1998), no.~2, 441--446.
  \MR{1415331}

\bibitem{MR1475774}
Alessandra Lunardi, \emph{Schauder estimates for a class of degenerate elliptic
  and parabolic operators with unbounded coefficients in {${\bf R}^n$}}, Ann.
  Scuola Norm. Sup. Pisa Cl. Sci. (4) \textbf{24} (1997), no.~1, 133--164.
  \MR{1475774}

\bibitem{Phung}
K.D. Phung, \emph{Doubling property for the laplaican and its applications
  (course chengdu 2007)}, https://www.idpoisson.fr/phung/phung\_j.pdf.

\bibitem{Plis}
A.~Pli\'s, \emph{On non-uniqueness in {C}auchy problem for an elliptic second
  order differential equation}, Bull. Acad. Polon. Sci. S\'er. Sci. Math.
  Astronom. Phys. \textbf{11} (1963), 95--100. \MR{153959}

\bibitem{MR4207950}
Yannick Sire, Susanna Terracini, and Stefano Vita, \emph{Liouville type
  theorems and regularity of solutions to degenerate or singular problems part
  {I}: even solutions}, Comm. Partial Differential Equations \textbf{46}
  (2021), no.~2, 310--361. \MR{4207950}

\bibitem{stuart2018stability}
Charles~A. Stuart, \emph{Stability analysis for a family of degenerate
  semilinear parabolic problems}, Discrete Contin. Dyn. Syst. \textbf{38}
  (2018), no.~10, 5297--5337. \MR{3834720}

\bibitem{tao2008weighted}
Xiangxing Tao and Songyan Zhang, \emph{Weighted doubling properties and unique
  continuation theorems for the degenerate {S}chr\"odinger equations with
  singular potentials}, J. Math. Anal. Appl. \textbf{339} (2008), no.~1,
  70--84. \MR{2370633}

\bibitem{Vessella}
S.~Vessella, \emph{Quantitative continuation from a measurable set of solutions
  of elliptic equations}, Proc. Roy. Soc. Edinburgh Sect. A \textbf{130}
  (2000), no.~4, 909--923. \MR{1776684}

\bibitem{wu2020carleman}
Bin Wu, Qun Chen, and Zewen Wang, \emph{Carleman estimates for a stochastic
  degenerate parabolic equation and applications to null controllability and an
  inverse random source problem}, Inverse Problems \textbf{36} (2020), no.~7,
  075014, 38. \MR{4121326}

\bibitem{Wu}
W.~Wu, Y.~Hu, Y.~Liu, and D.~Yang, \emph{Carleman estimates for degenerate
  parabolic equations with single interior point degeneracy and its
  applications}, Mathematical Control and Related Fields (2025).

\end{thebibliography}
%
%
%

\end{document}